\documentclass[12pt]{article}

\usepackage{graphicx} 
\usepackage{algorithm,algcompatible}
\usepackage{float}    
\usepackage{verbatim} 
\usepackage{amsmath}  
\usepackage{amssymb}  
\usepackage{subfig}   
\usepackage[colorlinks=true,citecolor=blue,linkcolor=blue,urlcolor=blue]{hyperref}
\usepackage{fullpage}
\usepackage{amsthm,mathtools}
\usepackage{enumerate}
\usepackage{paralist}
\usepackage{xspace}
\usepackage{caption}
\usepackage{bbm}
\usepackage{url}
\usepackage{mathrsfs}

\usepackage[colorinlistoftodos,prependcaption,textsize=tiny]{todonotes}
\presetkeys{todonotes}{inline}{}

\newtheorem{theorem}{Theorem}
\newtheorem{lemma}[theorem]{Lemma}
\newtheorem{proposition}[theorem]{Proposition}
\newtheorem{coro}[theorem]{Corollary}

\theoremstyle{definition}

\newcommand{\tf}[1]{\mathbf{#1}}
\newcommand{\Ah}{\widehat{A}}
\newcommand{\Ahat}{\Ah}
\newcommand{\Bhat}{\Bh}
\newcommand{\Bh}{\widehat{B}}
\newcommand{\Kh}{\widehat{\tf K}}

\newcommand{\Phixh}{\widehat{\tf \Phi}_x}
\newcommand{\Phixs}{\tf \Phi_x^\star}
\newcommand{\Phixc}{\tf \Phi_x^c}
\newcommand{\Phiuh}{\widehat{\tf \Phi}_u}
\newcommand{\Phius}{\tf \Phi_u^\star}
\newcommand{\Phiuc}{\tf \Phi_u^c}
\newcommand{\Phixtilde}{\tilde{\tf \Phi}_x}
\newcommand{\Phiutilde}{\tilde{\tf \Phi}_u}
\newcommand{\Dh}{\widehat{\tf{\Delta}}}
\newcommand{\trueA}{A_\star }
\newcommand{\trueB}{B_\star }

\newcommand{\RHinf}{\mathcal{RH}_\infty}

\DeclareMathOperator*{\Tr}{\mathbf{Tr}}
\DeclareMathOperator*{\argmin}{argmin}

\newcommand{\norm}[1]{\lVert #1 \rVert}
\newcommand{\bignorm}[1]{\left\lVert #1 \right\rVert}
\newcommand{\twonorm}[1]{\lVert #1 \rVert_{2}}

\newcommand{\ip}[2]{\ensuremath{\langle #1, #2 \rangle}}
\newcommand{\E}{\mathbb{E}}
\newcommand{\abs}[1]{\ensuremath{| #1 |}}

\renewcommand{\Pr}{\mathbb{P}}
\newcommand{\T}{\top}

\newcommand{\statedim}{n}
\newcommand{\inputdim}{d}
\newcommand{\hinf}{\mathcal{H}_\infty}
\newcommand{\htwo}{\mathcal{H}_2}
\newcommand{\lone}{\mathcal{L}_1}

\newcommand{\infnorm}[1]{\| #1 \|_\infty}
\newcommand{\hinfnorm}[1]{\| #1 \|_{\hinf}}
\newcommand{\htwonorm}[1]{\| #1 \|_{\htwo}}
\newcommand{\lonenorm}[1]{\| #1 \|_{\lone}}

\newcommand{\bightwonorm}[1]{\bignorm{#1}_{\mathcal{H}_2}}
\newcommand{\bighinfnorm}[1]{\bignorm{#1}_{\mathcal{H}_\infty}}

\newcommand{\R}{\mathbb{R}}

\newcommand{\calX}{\mathcal{X}}
\newcommand{\calU}{\mathcal{U}}

\newcommand{\calO}{\mathcal{O}}

\newcommand{\Toep}{\mathrm{Toep}}

\newcommand{\eps}{\varepsilon}

\numberwithin{theorem}{section}

\numberwithin{equation}{section}

\newcommand{\extendedsls}{Appendix~\ref{sec:all_sls_proofs}}
\newcommand{\extendedest}{Appendix~\ref{sec:payley_zygmund_proof}}
\newcommand{\imagewidth}{300px}

\title{\bf
Safely Learning to Control the Constrained Linear Quadratic Regulator
}

\author{Sarah Dean, Stephen Tu, Nikolai Matni, and Benjamin Recht
\thanks{S. Dean, S. Tu, N. Matni, and B. Recht are with the Department of Electrical Engineering and Computer Sciences, University of California, Berkeley, CA, 94709 USA ({email: dean\_sarah@berkeley.edu, stephent@berkeley.edu, nmatni@berkeley.edu, brecht@berkeley.edu})}%
}

\begin{document}

\maketitle
\pagestyle{plain}

\begin{abstract}
We study the constrained linear quadratic regulator with unknown dynamics,
addressing the tension between safety and exploration in data-driven control techniques.
We present a framework which allows for system identification through persistent excitation,
while maintaining safety by guaranteeing the satisfaction of state and input constraints.
This framework involves a novel method for synthesizing robust constraint-satisfying feedback controllers, leveraging newly developed tools from 
\emph{system level synthesis}.
We connect statistical results with cost sub-optimality bounds to give
non-asymptotic guarantees on both estimation and controller performance.
\end{abstract}

\section{Introduction}

Data-driven design has considerable potential in contemporary control systems where precise modeling of the dynamics is intractable, whether due to complex large-scale interactions or nonlinearities resulting from contact forces.
However, a large hurdle in the way of practical deployment is the question of maintaining safe operating conditions during the learning process, and furthermore guaranteeing safe execution using learned components.

Motivated by this issue, we study the data-driven design of a
controller for the \emph{constrained} Linear Quadratic Regulator (LQR) problem.
In constrained LQR, we design a controller
for a potentially unknown linear dynamical system that minimizes a given quadratic
cost, subject to the additional requirement that both the state and input stay within
a specified safe region. This is a problem that has received much attention
within the model predictive control (MPC) community.

For the LQR problem with no constraints, a natural method of exploration for
learning the dynamics is to excite the system by injecting white noise. When
safety is not an issue, this method is effective and recently Dean et
al.~\cite{dean2017sample} provide an end-to-end sample complexity on this
``identify-then-control'' scheme. However, this method of fails to consider safety, and the injected white noise may lead to constraint violation.

%
In this paper, we directly address the tension between exploration for learning
and safety, which are fundamentally at odds. We do this by
synthesizing a controller which simultaneously excites and regulates the system; we propose to learn by {additively} injecting bounded
noise to the {control} inputs computed by a safe controller. By leveraging the recently developed
system level synthesis (SLS) framework for control design~\cite{SysLevelSyn1}, we give
a computationally tractable algorithm which returns a controller that (a)
guarantees the closed loop system remains within the specified constraint set and
(b) ensures that enough noise can be injected into the system to obtain a statistical
guarantee on learning. To the best of our knowledge, our algorithm is the first
to simultaneously achieve both objectives. 
Furthermore, the controller synthesis is solved by a convex optimization problem whose feasibility is a certificate of safety, and no considerations of robust invariant sets are required. 

Our second contribution is to provide a sub-optimality bound on control
performance for constrained LQR. Using the same SLS framework, we quantify the
excess cost incurred by playing a controller designed on the uncertain dynamics
obtained from learning.
The sub-optimality depends on both the size of the uncertainty sets and a
type of constraint robustness cost gap of the optimal
constrained controller for the true system.  This allows us to provide the
first end-to-end sample complexity guarantee for the control of constrained
systems.




Estimation and control of the unconstrained LQR problem has been studied in the non-asymptotic setting~\cite{fiechter1997pac, dean2017sample}. However, the identification schemes rely on pure excitation and {system restarts}, which is not suitable for the constrained setting. 
On the other hand, the online learning literature simultaneously considers learning and control, where strategies are based on \emph{optimism in the face of uncertainty} \cite{abbasi2011regret} or \emph{Thompson Sampling} \cite{abeille17}. These approaches guarantee system estimation only up to optimal closed-loop equivalence, and do not consider safety. Building on a statistical result by Simchowitz et al.~\cite{simchowitz2018learning} which allows for non-asymptotic guarantees on parameter estimation from a single trajectory of a linear system, Dean et al.~\cite{dean2018regret} provide a robust online method that guarantees parameter estimation and stability throughout.

The design of controllers that guarantee robust constraint satisfaction has long been considered in the context of model predictive control~\cite{bemporad1999robust}, including methods that model uncertainty in the dynamics directly~\cite{kothare1996robust}, or model it as a bounded state disturbance for computational efficiency~\cite{mayne2005robust,goulart2006optimization}. 
Strategies for incorporating estimation of the dynamics include 
experiment-design inspired costs~\cite{larsson2011mpc},
decoupling learning from constraint satisfaction~\cite{aswani2013provably},
and set-membership methods rather than parameter estimation~\cite{tanaskovic2014adaptive,lorenzen2017adaptive}.
Due to the receding horizon nature of model predictive controllers, this literature relies on set invariance theory for infinite horizon guarantees~\cite{blanchini99}. Our framework considers the infinite horizon problem directly, and therefore we do not require computation of invariant sets.

Finally, the machine-learning community has begun to consider safety in
reinforcement learning, where much work positions itself as being for general
dynamical systems in lieu of providing statistical guarantees~\cite{berkenkamp15,berkenkamp2017safe,dalal18,chow18}. Some works assume the existence of an initial safe controller for learning
\cite{koller2018learning}, and 
robust MPC methods have been proposed to modify potentially unsafe learning inputs~\cite{wabersich2018linear}.
Our framework gives an alternative procedure for designing such a controller using coarse system estimates.  
Most similar
to this work is that of Lu et al.~\cite{lu2017safe}, who propose a method to
allow excitation on top of a safe controller, but consider only finite-time
safety and require
non-convex optimization to obtain formal guarantees. 

We fix an underlying linear dynamical system 
with full state observation,
\begin{align}\label{eq:dynamics}
x_{k+1} = \trueA x_k + \trueB u_k + w_k\:,
\end{align}
where we have initial condition $x_0 \in \R^\statedim$, sequence of inputs $\{u_0,u_1,\dots\} \subseteq \R^\inputdim$, and
disturbance process $\{w_0, w_1,\dots\} \subseteq \R^\statedim$. 
Later, we will specify assumptions on the disturbances, including independence over time, bounded second moment, and bounded $\ell_\infty$ norm.
The dynamics matrices $(\trueA, \trueB)$ are unknown.
We denote the errors on estimates of the system $(\Ahat,\Bhat)$ in terms of the $\ell_p\to\ell_p$ matrix operator norm as
\[\|\trueA-\Ahat\|_p:=\|\Delta_A\|_p:=\eps_{A,p},~~\|\trueB-\Bhat\|_p:=\|\Delta_B\|_p:=\eps_{B,p},\] 
In this work, we will focus on the errors for $p=2$ and $p=\infty$.

Furthermore, we assume some prior knowledge on the system dynamics,
in the form of initial estimates $(\Ahat_0, \Bhat_0)$ and uncertainty measures
$(\eps_{A,p}^0,\eps_{B,p}^0)$. 
Without such knowledge, guaranteeing safety in any form would be impossible.
We note that the initial estimates
may be coarse grained, and the goal of the learning procedure will be to refine this uncertainty prior to optimal control design.

The constrained LQR optimal control problem seeks to minimize the expected quadratic cost, subject to constraints on the state and input.
Before further detailing this objective, we develop notation and background on system level synthesis.

\subsection{System Level Synthesis}
Many approaches to optimal control for systems with constraints involve receding horizon control, where an \emph{open loop} finite-time trajectory is computed at each timestep; indeed, parameterizing optimal control problems by a state feedback controller generally leads to nonconvex optimization. As a motivating example, consider the static feedback law $u_k=K x_k$ applied to the linear system~\eqref{eq:dynamics}. Then we can write
\[x_k = \sum_{t=1}^k (\trueA+\trueB K)^{t} w_{k-t} + (\trueA+\trueB K)^{k}x_0\:,\]
and a similar expression for $u_k$. Then it is clear that convex constraints on $x_k$ and $u_k$ will be non-convex in $K$.
Instead, we can parametrize the problem in terms of convolution with the closed-loop system response,
\begin{align}\label{eq:phis}
x_k = \sum_{t=1}^{k+1} \Phi_x(t) w_{k-t}, \:u_k = \sum_{t=1}^{k+1} \Phi_u(t) w_{k-t}
\end{align}
where we defined $w_{-1}=x_0$ the fixed initial condition. 
We note that the expression above is valid for any linear dynamic controller, i.e. any controller which is a linear function of the state and its history.
Since the relation in~\eqref{eq:phis} is linear, convex constraints on state and input translate to convex constraints on the system response elements. 
The \emph{system level synthesis} (SLS) framework shows that for any elements $\{\Phi_x(t),\Phi_u(t)\}$ constrained to obey, for all $k \geq 1$,
\begin{equation*}
\Phi_x(k+1) = \trueA \Phi_x(k) + \trueB \Phi_u(k) \:, \:\: \Phi_x(1) = I \:,
\end{equation*}
there exists a feedback controller that achieves the desired system responses~\eqref{eq:phis}. The state-feedback parameterization result in Theorem 1 of Wang et al.~\cite{SysLevelSyn1} formalizes this principle, and therefore any optimal control problem over linear systems can be cast as a constrained optimization problem over system response elements. 
We remark that similar observations have been used for constrained state feedback control in the finite horizon setting~\cite{goulart2006optimization}.

In what follows, we use boldface letters to denote transfer functions and signals, e.g. $\tf{\Phi}_x(z) = \sum_{k = 1}^\infty\Phi_x(k) z^{-k}$ and $\tf{x}(z) = \sum_{k = 0}^\infty x_k z^{-k}$.
Under this notation, the affine constraints can be rewritten as
\begin{equation*}
\begin{bmatrix} zI - \trueA & - \trueB \end{bmatrix} \begin{bmatrix} \tf \Phi_x \\ \tf \Phi_u \end{bmatrix} = I \:,
\end{equation*}
and the corresponding control law $\tf u = \tf K \tf x$ is given by $\tf K = \tf \Phi_u \tf \Phi^{-1}_x$.

\subsection{Notation}

In this paper, we restrict our attention to the function space $\RHinf$, consisting of discrete-time stable
matrix-valued transfer functions. We use $\frac{1}{z} \RHinf$ to denote the set
of transfer functions $\tf G$ such that $z \tf G \in \mathcal{RH}_\infty$. We further define
$$\RHinf(C, \rho)
 := \Big\{ \tf M = \sum_{k=0}^\infty M(k) z^{-k} ~|~ \norm{ M(k)}_2 \leq C \rho^k,~k = 0, 1, 2, ...\Big\}$$
 for transfer functions that satisfy a certain decay rate in the spectral norm of their impulse response elements.

When working with transfer functions and signals, we will denote the coefficient of the term of degree $k$ as $\tf G[k]=G(k)$ and  $\tf x[k]=x_k$. We will also denote $G[k:1]$ as the block row vector of system response elements of $\tf G$
 \[G[k:1] = \begin{bmatrix} G(k) & \dots & G(1) \end{bmatrix}\:. \]

As is standard, we let $\norm{x}_p$ denote the $\ell_p$-norm of a vector $x$.
For a matrix $M$, we let $\norm{M}_{p}$ denote its $\ell_p \to \ell_p$ operator norm. 
We will consider the $\htwo$, $\hinf$, and $\lone$ norms, which are infinite horizon analogs of the Frobenius, spectral, and $\ell_\infty\to\ell_\infty$ operator norms of a matrix, respectively:
$$\htwonorm{\tf M} = {\sqrt{\sum_{k=0}^\infty \norm{ M(k)}_F^2 }},~~
\hinfnorm{\tf M} = \sup_{\twonorm{\tf w}=1} \: \twonorm{\tf M\tf w},~~
\lonenorm{\tf M} = \sup_{\infnorm{\tf w}=1} \: \infnorm{\tf M\tf w}$$
%
%

Finally, for two numbers $a, b$, we let $a \lesssim b$ (resp. $a \gtrsim b$) denote that there exists
an absolute constant $C > 0$ such that $a \leq C b$ (resp. $a \geq C b$).

\subsection{Optimal Control Problem}
We now describe the constrained optimal control problem
that we would want to solve given perfect knowledge of the system dynamics. 
First, we consider the expected infinite horizon quadratic cost for the system $(\trueA, \trueB)$ in feedback with $\tf K$:
\[J(\trueA, \trueB, \tf K)^2 :=\frac{1}{\sigma^2} \lim_{T \to \infty} \frac{1}{T} \sum_{k=0}^{T-1} \E_{w}[ x_k^\T Q x_k + u_k^\T R u_k]\:.\]
We assume that $\{w_k\}$ is any distribution that
satisfies $\E[w_k]=0$ and $\E[w_kw_k^\T] = \sigma^2 I$ and is independent across time, i.e., $w_k\perp w_\ell$ for $\ell\neq k$.
Note that
any distribution satisfying these constraints induces the same expected quadratic cost, so it is unnecessary
to specify a specific distribution.

Next, we consider polytopic constraint sets of the form
\[\calX := \{x~:~F_x x\leq b_x\},~~\calU := \{u~:~F_u u\leq b_u\}\:.\]
We will constrain the state and input trajectories to lie within these sets for any possible disturbance sequence
$\{w_k\}$ satisfying $\norm{w_k}_{\infty} \leq \sigma_w$ for all $k\geq 0$.
Putting the cost and the constraints together, the optimal control problem that acts as our baseline is:
\begin{align}
\begin{split} \label{eq:constrained_lqr_stoch}
  \min_{\tf u \in\mathcal{K}} ~& J(\trueA, \trueB, \tf K)\\
  \text{s.t.}~& x_0~\text{fixed};~~  F_x x_k \leq b_x,~~ F_u u_k \leq b_u \\
    &\qquad\qquad\forall k \:, \forall \{w_k : \norm{w_k}_\infty \leq \sigma_w\} \:.
  \end{split}
\end{align}
Above, we let the set $\mathcal{K}$ enumerate all inputs that result from linear dynamic stabilizing feedback controllers
for $(\trueA, \trueB)$ of the form $\tf u=\tf K \tf x$. This is made possible by the system level synthesis framework described in the previous section. 
%

As we show in the following section, the optimal control problem given in \eqref{eq:constrained_lqr_stoch} is a convex, but infinite-dimensional
problem. It is an idealized baseline to
compare our actual solutions to; our sub-optimality guarantees will be with respect
to the optimal cost achieved by this problem. It 
is a relevant baseline, since it optimizes for average case performance but ensures
safety for the worst-case behavior, consistent with MPC literature~\cite{mayne2005robust,oldewurtel2008tractable}. We remark that an alternative to
\eqref{eq:constrained_lqr_stoch} is to replace the worst case constraint behavior with probabilistic
chance constraints~\cite{farina2016stochastic}. We do not work with chance constraints because they are generally
difficult to directly enforce on an infinite horizon; arguments around recursive feasibility
using robust invariant sets are common in the MPC literature to deal with this issue.

\section{Constraint-Satisfying Control}
We begin by formulating a method for robustly operating a system while maintaining safety. 
First we describe a system level synthesis approach to the constrained LQR problem, and then we propose a modification which makes it robust to uncertainties in system dynamics. 
Finally, we outline a reduction to tractable synthesis via a finite-dimensional optimization problem over a finite impulse response (FIR) approximation.

\subsection{A System Level Approach}
Using the SLS formulation, we define an optimization problem that solves the optimal control problem presented in the section above.

\begin{proposition} \label{prop:sls_equivalence}
The following convex optimization problem solves the optimal control problem~\eqref{eq:constrained_lqr_stoch}.
\begin{align}
\min_{\tf\Phi_x,\tf\Phi_u\in \frac{1}{z} \RHinf} ~ & \left\|\begin{bmatrix}Q^{1/2} & \\ & R^{1/2}\end{bmatrix} \begin{bmatrix}\tf\Phi_x \\ \tf\Phi_u\end{bmatrix}\right\|_{\htwo} \label{eq:sls_lqr}\\
\mathrm{s.t.}~& \begin{bmatrix} zI - \trueA & -\trueB \end{bmatrix} \begin{bmatrix}\tf\Phi_x \\ \tf\Phi_u\end{bmatrix} = I,\notag\\
& G_x(\tf\Phi_x;k) \leq b_x, ~~ G_u(\tf\Phi_u;k) \leq b_u~~\forall ~k\geq 1\:. \notag
\end{align}
where the elements of the constraint functions are defined as
\begin{align*}
\begin{split}
G_x(\tf\Phi_x;k)_j := F_{x,j}^\top \Phi_x(k+1)x_0 +  \sigma_w \|F_{x,j}^\top \Phi_x[k:1]\|_1\:,  \\
 G_u(\tf\Phi_u;k)_j := F_{u,j}^\top \Phi_u(k+1)x_0 +  \sigma_w\|F_{u,j}^\top \Phi_u[k:1]\|_1\:,
\end{split}
\end{align*}
with $j$ indexing the rows of $F_x$ and $F_u$ and entries of $G_x$ and $G_u$.
\end{proposition} \vspace{0.5em}
Examining the form of the optimization problem above, we see that the expected quadratic cost transforms into a system $\htwo$ norm, while the worst-case polytopic constraints on state and input become closed-form polytopic constraints on the system response.
For a controller $\tf K =\tf\Phi_u(\tf\Phi_x)^{-1}$, note that the LQR cost is equivalent to
\[J(\trueA,\trueB,\tf K) = \left\|\begin{bmatrix}Q^{1/2} & \\ & R^{1/2}\end{bmatrix} \begin{bmatrix}\tf\Phi_x \\ \tf\Phi_u\end{bmatrix}\right\|_{\htwo}\:.\]
Lastly, we remark here that the feasibility of the convex synthesis problem in~\eqref{prop:sls_equivalence} for an initial condition $x_0$ implies that $x_0$ is a member of a robust control invariant set.

\begin{proof}
By the state-feedback parameterization result in Theorem 1 of~\cite{SysLevelSyn1}, the SLS parametrization encompasses all internally stabilizing state-feedback controllers acting on the true system $(\trueA,\trueB)$.
Thus, it is necessary only to show that the optimization problem in~\eqref{eq:sls_lqr} is consistent with that of~\eqref{eq:constrained_lqr_stoch} under the system level parametrization. The equivalence between the LQR cost and the $\htwo$ system norm is standard and we refer to the Appendix of \cite{dean2017sample}, which presents this reformulation in terms of system responses for Gaussian process noise. The argument applies unchanged to any distribution satisfying our assumptions.

Therefore, it remains to consider the inequality constraints.
Because the constraints must be satisfied robustly, it is equivalent to consider
\begin{align*}
&\max_{\{w_k\}} F_x \sum_{t=1}^{k+1} \Phi_x(t) w_{k-t} \leq b_x
\iff  F_x \Phi_x(k+1)x_0 +  \max_{\{w_k\}} F_x \sum_{t=1}^{k} \Phi_x(t) w_{k-t} \leq b_x\:.
\end{align*}
Then considering elements in the second term with each $j$ indexing the rows of $F_x$,
\begin{align*}
\max_{\{w_k\}} &F_{x,j}^\top \sum_{t=1}^{k} \Phi_x(t) w_{k-t} =  \sum_{t=1}^{k} \max_{\|w\|_\infty\leq \sigma_w}  F_{x,j}^\top \Phi_x(t) w 
=\sum_{t=1}^{k} \sigma_w\|F_{x,j}^\top \Phi_x(t)\|_1 = \sigma_w\|F_{x,j}^\top \Phi_x[k:1]\|_1 \:.
\end{align*}
Thus the inequality constraint on the function $G_x(\Phi_x;k)$ is an equivalent condition.
A similar computation holds for the input constraint.
\end{proof}
{
We note the appearance of the row-wise $\ell_1$ norm over the multiplication of the the system response elements with the constraint matrix.
The expression
can be understood as an analog to the $\ell_\infty\to\ell_\infty$ operator norm, mediated by the shape of the polytope. 
In the next section when we introduce robustness to system dynamics, the $\lone$ system norm will come into play for similar reasons.
}

\subsection{Robust Control} \label{sec:robust_control}
We are further motivated to reformulate the optimal control problem in terms of system responses so that we can transparently consider uncertainties in the dynamics. We are interested in controller synthesis under model errors, where only nominal estimates of the system are known. Suppose that the system responses $\widehat{\tf\Phi}_x$, $\widehat{\tf\Phi}_u$ are designed on the estimated system. Then

\[\begin{bmatrix} zI - \Ahat & -\Bhat \end{bmatrix} \begin{bmatrix}\widehat{\tf\Phi}_x \\ \widehat{\tf\Phi}_u\end{bmatrix} = I\:,\text{and define}~\widehat{\tf\Delta} := \begin{bmatrix} \Delta_A & \Delta_B \end{bmatrix} \begin{bmatrix} \widehat{\tf\Phi}_x \\ \widehat{\tf\Phi}_u \end{bmatrix}\:.\]
The model mismatch impacts the closed-loop behavior of the true system in feedback with the controller $\widehat {\tf K} = \widehat{\tf\Phi}_u(\widehat{\tf\Phi}_x)^{-1}$ in a transparent way. Supposing that $ \|\widehat{\tf\Delta}\|_{\hinf}<1$, it achieves the system response and bounded cost:
\begin{align}\label{eq:delta_perf_bound}
\begin{bmatrix}\widehat{\tf\Phi}_x \\ \widehat{\tf\Phi}_x\end{bmatrix}(I + \widehat{\tf\Delta})^{-1}\:,~~J(\trueA, \trueB, \Kh) \leq \frac{1}{1-\|\widehat{\tf\Delta}\|_{\hinf}}
J(\Ah, \Bh, \Kh)\:.
\end{align}
This follows from the robust stability result in Theorem 2 of Matni et al.~\cite{virtual} and the sub-multiplicativity of the $\htwo$ and $\hinf$ norms. Note that $\|\Dh\|<1$ is a sufficient condition for the existence of the inverse $(I+\Dh)^{-1}$ for any induced norm $\|.\|$ by the small gain theorem. 

This expression provides an upper bound on the cost achieved for a controller designed system estimates which depends only on the system estimates and the size of the error system $\widehat{\tf\Delta}$.
Motivated by this bound, consider the following robust optimization problem:
\begin{align}
\widehat J(\gamma,\tau):=\min_{\tf\Phi_x,\tf\Phi_u\in\frac{1}{z} \RHinf} ~ &\frac{1}{1-\gamma} J(\Ah, \Bh, \tf K) \label{eq:sls_lqr_robust} \\
\mathrm{s.t.}~& \begin{bmatrix} zI - \Ahat & -\Bhat \end{bmatrix} \begin{bmatrix}\tf\Phi_x \\ \tf\Phi_u\end{bmatrix} = I,\notag \\
& \sqrt{2} \left\| \begin{bmatrix}\eps_{A,2}\tf\Phi_x \\ \eps_{B,2}\tf\Phi_u\end{bmatrix}\right\|_{\hinf} \leq \gamma, ~ \left\| \begin{bmatrix}\eps_{A,\infty}\tf\Phi_x \\ \eps_{B,\infty}\tf\Phi_u\end{bmatrix}\right\|_{\lone} \leq \tau, \notag\\
&G_x^\tau(\Phi_x;k) \leq b_x, ~~ G_u^\tau(\Phi_u;k) \leq b_u~~\forall ~k\geq 1\:. \notag
\end{align}
where $0\leq\gamma,\tau\leq 1$ are fixed parameters and for each $j$ and $k$,
\begin{align*}
\begin{split}
G_x^\tau(\tf\Phi_x;k)_j := G_x(\tf\Phi_x;k)_j+\frac{\tau }{1-\tau}\max(\sigma_w,\|x_0\|_{\infty})\|F_{x,j}^\top\Phi_x[k+1:1]\|_1 
\:,
\end{split}
\end{align*}
and similarly for $G_u^\tau(\tf\Phi_u;k)_j$.

In this problem, $\gamma$ bounds the increase in the $\htwo$ cost due to the dynamics uncertainty, while $\tau$ determines the increase in the state and input values with respect to the constraints.
In fact, both values can be related to bounding an enlarge noise process $\tilde {\tf w} = (I+\Dh)^{-1}\tf w$ driving the system.
\begin{theorem}\label{thm:robust_constraint_satisfaction}
Any controller designed from a feasible solution to the robust control problem~\eqref{eq:sls_lqr_robust} for any $0\leq \gamma,\tau<1$ will stabilize the true system.
Furthermore,
the state and input constraints will be satisfied.
\end{theorem}
\begin{proof}
First, note that by the system norm constraints,
\begin{align}
\begin{split}\label{eq:delta_bound}
&\|\Dh\|_{\hinf} \leq \sqrt{2} \left\| \begin{bmatrix}\eps_{A,2}\tf\Phi_x \\ \eps_{B,2}\tf\Phi_u\end{bmatrix}\right\|_{\hinf}<1~,~~\|\Dh\|_{\lone} \leq \left\| \begin{bmatrix}\eps_{A,\infty}\tf\Phi_x \\ \eps_{B,\infty}\tf\Phi_u\end{bmatrix}\right\|_{\lone}<1\:.
\end{split}
\end{align}
Then by~\eqref{eq:delta_perf_bound},
the true system trajectory will be given by the stable system responses
\[\tf x = \tf \Phi_x (I+\Dh)^{-1} \tf w,\quad \tf u = \tf \Phi_u (I+\Dh)^{-1} \tf w\:.\]
Therefore, the state constraints are satisfied as long as
\begin{align*}
b_x &\geq \max_{\tf w} F_x \tf x[k] = \max_{\tf w} F_x (\tf \Phi_x (I+\Dh)^{-1} \tf w )[k]\\
& = \max_{\tf w} F_x (\tf \Phi_x \tf w)[k] - F_x (\tf \Phi_x \Dh(I+\Dh)^{-1} \tf w)[k]\:.
\end{align*}
The first term reduces to $G_x(\Phi_x;k)$ as in the non-robust case. Because information about $\Dh$ is not known, we resort to a sufficient condition to bound the second term, letting $\tilde{\tf w}= \Dh(I+\Dh)^{-1} \tf w$,
\begin{align*}
|F_{x,j}^\top (\tf \Phi_x \tilde{\tf w})[k]|
 &\leq \|F_{x,j}^\top \Phi_x[k+1:1]\|_1 \|\tilde{\tf w}\|_{\infty}  \\
&\leq \|F_{x,j}^\top \Phi_x[k+1:1]\|_1 \frac{\|\Dh\|_{\lone}}{1-\|\Dh\|_{\lone}} \max(\sigma_w,\|x_0\|_\infty)\\
&\leq \|F_{x,j}^\top \Phi_x[k+1:1]\|_1 \frac{\tau}{1-\tau} \max(\sigma_w,\|x_0\|_\infty) \:.
\end{align*}
Consequently, a sufficient condition for satisfying state constraints is to have for $j$ indexing rows of $F_x$ and elements of $b_x$,
\begin{align*}
G_x(\tf\Phi_x;k)_j+\frac{\tau}{1-\tau}\max(\sigma_w,\|x_0\|_\infty)\|F_{x,j}^\top\Phi_x[k+1:1]\|_1 \leq b_{x,j}\:.
\end{align*}
Therefore, the constraints on $\tf\Phi_x$ imply that the state constraints are satisfied. Similar logic shows that the constraints on $\tf\Phi_u$ imply that the input constraints are satisfied.
\end{proof}
We remark on differences between the presented constraint-tightening and approaches common in the MPC literature. In both cases, the uncertainty in the dynamics  induces an enlarged disturbance process. 
Common in the MPC literature, the additive disturbance approximation assigns $\tilde w_k = w_k + \Delta_A x_k + \Delta_B u_k$, or equivalently $\tilde{\tf w} = \tf w + \Delta_A \tf x + \Delta_B \tf u$. 
Then $\|\tilde{\tf w}\|_\infty$ can be bounded using the state and input constraint sets.
Because this bound degrades as the constraint sets increase in size, this strategy can lead to counterintuitive behavior.
On the other hand, our approach writes
$\tilde{\tf w} =\tf w + \widehat{\tf\Delta}(I+\widehat{\tf\Delta})^{-1}\tf w$.
While bounding $\|\tilde{\tf w}\|_\infty$ in this setting does not depend on the constraint set, it is affected by the size of the initial condition.
Further comparison between the two approaches is presented in Appendix~\ref{sec:alternate_constraints}.

\subsection{Finite Dimensional Reduction}

To make controller synthesis tractable, we can solve a finite approximation to optimization problem~\eqref{eq:sls_lqr_robust} wherein we only optimize over the first $L$ impulse response elements of $\tf\Phi_x$ and $\tf\Phi_u$, treating them as finite impulse response (FIR) filters.  We show that in this setting, the optimization variables and constraints admit finite-dimensional representations.
We first reformulate the constraints. Starting with the affine constraint, we have for $k=1,...,L-1$
\begin{align}\label{eq:FIR_constr_dyn}
\Phi_x(1) = I,~&\Phi_x(k+1) = \Ahat \Phi_x(k)+\Bhat \Phi_u(k)\\
&V = \Ahat \Phi_x(L)+\Bhat \Phi_u(L) \nonumber\:,
\end{align}
where we will also optimize over $V$, a term which captures the tail of the system responses that we ignore in the synthesis.


Next, considering the system norm constraints, the $\hinf$ norm can be reduced to a compact SDP over $\Phi_x,\Phi_u, \gamma$ as in Theorem 5.8 of Dumitrescu~\cite{dumitrescu2007positive}, described explicitly for this setting in Appendix G.3 of Dean et al.~\cite{dean2018regret}.
For the $\lone$ norm, the constraint becomes an $\ell_\infty\to\ell_\infty$ operator norm bound,
\begin{align}
\begin{split}\label{eq:FIR_constr_inf}
\left\| \begin{bmatrix}\eps_{A,\infty}\tf\Phi_x \\ \eps_{B,\infty}\tf\Phi_u\end{bmatrix}\right\|_{\lone}
	= \left\| \begin{bmatrix}\eps_{A,\infty}\Phi_x[L:1] \\ \eps_{B,\infty}\Phi_u[L:1]\end{bmatrix}\right\|_{\!\mathrlap{\infty}} + \|V\|_\infty \leq \tau \:,
\end{split}
\end{align}
where the tail variable $V$ enters transparently.
For $k=1,\dots, L-1$, the inequality constraints on $G^\tau_x(\Phi_x;k)$ and $G^\tau_u(\Phi_u;k)$ remain. For any $k\geq L$, the expression reduces to, for each $j$,
\begin{align} \label{eq:FIR_constr_stateinput}
\begin{split}
& \Big(\sigma_w+\frac{\tau \max(\sigma_w,\|x_0\|_\infty)}{1-\tau}\Big) \|F_{x,j}^\top \Phi_x[L:1]\|_1  \leq b_{x,j},\\& \Big(\sigma_w+\frac{\tau \max(\sigma_w,\|x_0\|_\infty)}{1-\tau}\Big) \|F_{u,j}^\top \Phi_u[L:1]\|_1  \leq b_{u,j}\:.
\end{split}
\end{align}

Therefore, the synthesis problem becomes, for any fixed $0\leq \gamma,\tau<1$,
\begin{align}
\min_{\Phi_x,\Phi_u,V} ~ &
\frac{1}{1-\gamma}\sum_{t=0}^{L} \Tr(\Phi_x(t)^\top Q \Phi_x(t) + \Phi_u(t)^\top R \Phi_u(t)) \notag\\
\mathrm{s.t.}~& \eqref{eq:FIR_constr_dyn},~\mathrm{SDP}(\Phi_x,\Phi_u, \gamma-\|V\|_2),~\eqref{eq:FIR_constr_inf},~\eqref{eq:FIR_constr_stateinput}\:,  \label{eq:sls_FIR}\\
&1\leq k\leq L-1:
G^\tau_x(\Phi_x;k) \leq b_{x},~ G^\tau_u(\Phi_u;k) \leq b_{u}.\notag
\end{align}
This is a finite dimensional SDP.
The controller given by $\tf K=\tf\Phi_u\tf\Phi_x^{-1}$ can be written in an equivalent state-space realization $(A_K, B_K, C_K, D_K)$ via Theorem 2 of Anderson et al.~\cite{anderson17}.



\section{Suboptimality Guarantees}

How much is control performance degraded by uncertainties about the dynamics? In this section, we derive a sub-optimality bound which answers this question for the constrained LQR problem. First, consider the addition of an outer minimization over $\gamma$ and $\tau$:\footnote{
	The objective $\eqref{eq:sls_lqr_robust}$ is unimodal in $\gamma,\tau$ individually, and therefore this outer minimization can be achieved by searching over the box $[0,1)\times [0,1)$. For less computational complexity, the minimization need only be over a single outer variable: $\max(\gamma,\tau)$. In this case, the sub-optimality bound will retain the same flavor, but the norm distinctions between cost and constraints will be less clear.
}
\begin{align}
\begin{split} \label{eq:sls_lqr_robust_tau}
\min_{\gamma,\tau} ~ \widehat J(\gamma,\tau) \:.
\end{split}
\end{align}
Denote the solution to the true optimal control problem as  $(\tf \Phi_x^\star  , \tf\Phi_u^\star  )$, then define $\tf K_\star =\tf \Phi^\star  _u{\tf\Phi^\star_x}^{-1}$ and $J_\star =J(A_\star , B_\star , \tf K_\star )$. Additionally, define constants related to the optimal system norm and the dynamics uncertainties:
\[ \zeta_2 = \left\| \begin{bmatrix}\eps_{A,2}\tf\Phi_x^\star \\ \eps_{B,2}\tf\Phi_u^\star\end{bmatrix}\right\|_{\hinf},~~ \zeta_\infty = \left\| \begin{bmatrix}\eps_{A,\infty}\tf\Phi_x^\star \\ \eps_{B,\infty}\tf\Phi_u^\star\end{bmatrix}\right\|_{\lone}\:. \]

Before stating the suboptimality result, we define a robust version of the optimal constrained controller.
Define the doubly robust constraint sets
\begin{align*}\bar G_x^{\zeta}(\tf\Phi_x;k)_j =  
F_{x,j}^\top \Phi(k+1)x_0 &+ \frac{\sigma_w}{1-\zeta_\infty} \|F_{x,j}^\top \Phi_x[k:1]\|_1
\\ &+ {\frac{2\zeta_\infty\max(\sigma_w, \|x_0\|_\infty)}{1-2\zeta_\infty}\|F_{x,j}^\top \Phi_x[k+1:1]\|_1}\:,
\end{align*}
and similarly for $\bar G_u^\zeta(\tf\Phi_u;k)_j$. 
Note that these constraint sets resemble the robust constraint sets with $\tau=2\zeta_\infty$ and an enlarged noise process with $\sigma_w\leftarrow\frac{\sigma_w}{1-\zeta_\infty}$.
Then the optimal robustly constrained controller is defined as
\begin{align}
(\tf\Phi^c_x, \tf\Phi^c_u) \in \argmin_{\tf\Phi_x,\tf\Phi_u\in\frac{1}{z} \RHinf} ~ & \left\|\begin{bmatrix}Q^{1/2} & \\ & R^{1/2}\end{bmatrix} \begin{bmatrix}\tf\Phi_x \\ \tf\Phi_u\end{bmatrix}\right\|_{\htwo} \label{eq:sls_lqr_rob_constraint}\\
\mathrm{s.t.}~& \begin{bmatrix} zI - \trueA & -\trueB \end{bmatrix} \begin{bmatrix}\tf\Phi_x \\ \tf\Phi_u\end{bmatrix} = I,\notag\\
&\left\| \begin{bmatrix}\eps_{A,2}\tf\Phi_x \\ \eps_{B,2}\tf\Phi_u\end{bmatrix}\right\|_{\hinf} \leq \zeta_2, ~ \left\| \begin{bmatrix}\eps_{A,\infty}\tf\Phi_x \\ \eps_{B,\infty}\tf\Phi_u\end{bmatrix}\right\|_{\lone} \leq \zeta_\infty,\notag \\
& \bar G_x^\zeta(\tf\Phi_x;k)   \leq b_x,~~\bar G_u^\zeta(\tf\Phi_u;k)\leq b_u~~\forall ~k\geq 1\:. \notag
\end{align}
with $\tf K_c=\tf\Phi_u(\tf\Phi_x^c)^{-1}$.
Note that this optimization problem designs a controller which has similar closed-loop behavior to the true optimal controller (i.e. system norms of similar size), but satisfies more stringent state and input constraints. 
The relative \emph{robustness cost gap} is defined as
\begin{align} \label{eq:gap}
M_{\zeta} = \frac{J(\trueA,\trueB,\tf K_c)-J(\trueA,\trueB,\tf K_\star)}{J(\trueA,\trueB,\tf K_\star)}\:.
\end{align}
Now we are ready to state a data independent bound on the sub-optimality of robust controllers synthesized on estimated dynamics.

\begin{theorem}\label{thm:suboptimality}
Suppose that the robust optimal constrained controller problem is feasible. 
As long as $\zeta_2 \leq \frac{1}{4\sqrt{2}}$ and
$\zeta_\infty \leq \frac{1}{2}$, we have that the cost achieved by $\Kh=\widehat{\tf \Phi}_u\widehat{\tf\Phi}_x^{-1}$ synthesized from the minimizers of~\eqref{eq:sls_lqr_robust_tau} satisfies
  \[\frac{J(\trueA,\trueB,\widehat{\tf K}) - J_\star}{J_\star} \leq4  \sqrt{2}(1+M_{\zeta}) (\eps_{A,2}+\eps_{B,2}\|\tf K_\star \|_{\hinf})\|\tf\Phi^\star_x\|_{\hinf}+M_{\zeta}\:.\]
\end{theorem} 
This result is stated in terms of quantities related to the unknown true system, with the goal of highlighting properties of systems that make them harder or easier to robustly control.
We see that the bound grows with the $\hinf$ norm of the optimal controller and closed-loop response. 
It also grows with the robustness cost gap $M_{\zeta}$.

Remark that $M_{\zeta}$ is be difficult to characterize analytically;  in general, it requires checking the boundaries of the robust polytopic constraints.
Once errors $\eps_A$ and $\eps_B$ are small enough that the optimal system response $(\Phixs,\Phius)$ satisfies the doubly robust constraints, we will have $M_{\zeta}=0$. If the optimal controller saturates its constraints, then $M_{\zeta}$ will be nonzero for any nonzero estimation errors. 
In Section~\ref{sec:experiments} (Figure~\ref{fig:tradeoff}), we numerically characterize this optimality gap for a double integrator example.


\begin{proof}
Using~\eqref{eq:delta_perf_bound} along with the norm bounds~\eqref{eq:delta_bound} and the constraints in optimization problem~\eqref{eq:sls_lqr_robust_tau},
\begin{align*}
J(\trueA,\trueB,\widehat{\tf K}) &\leq
\frac{1}{1-\widehat\gamma} J(\Ah, \Bh, \Kh) \:.
\end{align*}
Next, we will connect the optimal system response to the estimated system by constructing a feasible solution to the robust optimization problem. 
We will use the following lemma
\begin{lemma} \label{lem:feasible_soln_bound}
Under the conditions of Theorem~\ref{thm:suboptimality},
we have that the following is a feasible solution to~\eqref{eq:sls_lqr_robust_tau}
\begin{align*}
\tf{\tilde{\Phi}}_x = \tf\Phi^c_x(I-\tf\Delta)^{-1},~ \tf{\tilde{\Phi}}_u =\tf\Phi^c _u(I-\tf\Delta)^{-1},
~\tilde\gamma=\frac{\sqrt{2}\zeta_2}{1-\sqrt{2}\zeta_2},~~\tilde\tau=\frac{\zeta_\infty}{1-\zeta_\infty}\:.
\end{align*}
where we define ${\tf\Delta} := -\begin{bmatrix} \Delta_A & \Delta_B \end{bmatrix} \begin{bmatrix} \tf\Phi_x^c \\ \tf\Phi_u^c \end{bmatrix}$.
\end{lemma}
The proof of Lemma~\ref{lem:feasible_soln_bound} follows by checking that the proposed solution satisfies all the constraints and is presented in \extendedsls.
Applying Lemma~\ref{lem:feasible_soln_bound},
\begin{align*}
\frac{1}{1-\widehat\gamma} J(\Ah, \Bh, \Kh)  \leq \frac{1}{1-\tilde\gamma}J(\Ah, \Bh, \tf K_c )\:.\end{align*}
This is true because $(\Kh,\widehat\gamma)$ is the optimal solution to~\eqref{eq:sls_lqr_robust_tau}, so objective function with feasible $(\tilde{\tf{\Phi}}_u\tilde{\tf{\Phi}}_x^{-1}=\tf K_c,\tilde\gamma)$ is an upper bound. Then we have
\begin{align*}
J(\trueA,\trueB,\widehat{\tf K}) &\leq  \frac{1}{1-\tilde\gamma} J(\Ah, \Bh, \tf K_c ) \\
&\leq  \frac{1}{1-\tilde\gamma} \frac{1}{1-\|{\tf\Delta}\|_{2}} J(A_\star , B_\star , \tf K_c ) \\
&\leq  \frac{1}{1-\frac{\sqrt{2}\zeta}{1-\sqrt{2}\zeta}} \frac{1}{1-\sqrt{2}\zeta} (1+M) J(A_\star , B_\star , \tf K_\star ) \\
&\leq \left(1+ 4\sqrt{2}\zeta_2(1+M)+M\right) J(A_\star , B_\star , \tf K_\star ) \:.
\end{align*}
The second inequality follows from an application of~\eqref{eq:delta_perf_bound} with the roles of the nominal and true systems switched. The final follows from bounding $\|\tf\Delta\|_2$ by $\sqrt{2}\zeta$ and noticing that $\frac{x}{1-x}\leq 2x$ for $0\leq x\leq\frac{1}{2}$, where we set $x=2\sqrt{2}\zeta_2$. 
Finally, we note that
\[\zeta_2 = \left\| \begin{bmatrix}\eps_{A,2}I \\ \eps_{B,2}\tf K_\star\end{bmatrix}\tf\Phi_x^\star\right\|_{\hinf}\leq (\eps_{A,2}+\eps_{B,2}\|\tf K_\star \|_{\hinf})\|\tf\Phi^\star_x\|_{\hinf} \:.\]
\end{proof}

Here, we briefly remark that a similar sub-optimality bound can be derived for the finite problem in~\eqref{eq:sls_FIR}. In short, controllers synthesized by minimizing over $\gamma$ and $\tau$
will satisfy a sub-optimality bound of the form in Theorem~\ref{thm:suboptimality} with an additional factor due to the FIR truncation. The formal statement and proof of this result are deferred to \extendedsls, {but we highlight here that the cost penalty incurred due to FIR approximation decays exponentially in the horizon $L$ over which the approximation is taken.}

\section{Learning with Control}

Finally, we connect the previous results on robust control with system estimation.
We adopt control actions that both keep the system safe and provide excitation,
\begin{align}\label{eq:control_est}
\mathbf{u} = \mathbf{K} \mathbf{x} + \mathbf{\eta}
\end{align}
where each $\eta=(\eta_0,\eta_1,\dots)$ is stochastic and $\ell_\infty$-bounded,
i.e. $\norm{\eta_k}_{\infty} \leq \sigma_\eta$. Given a trajectory sequence $\{(x_k, u_k)\}_{k=0}^{T}$, we propose to learn the
dynamics $(\trueA, \trueB)$ via least-squares regression on a trajectory of length $T$:
\begin{align}
\label{eq:ls_problem}
(\Ah, \Bh) \in \argmin_{(A,B)} \sum_{k = 0}^{T - 1} \frac{1}{2} \|Ax_k + Bu_k - x_{k + 1}\|_2^2.
\end{align}
In what follows, we will prove a statistical rate on the least-squares estimate
$(\Ah, \Bh)$ in terms of the system response and the trajectory length.

The bulk of the proof for the statistical rate comes from a general theorem regarding
linear-response time series data from Simchowitz et al.~\cite{simchowitz2018learning}.
Recently, this proof was adopted by Dean et al.~\cite{dean2018regret} to show a rate of estimation
in the setting given by \eqref{eq:control_est} when both $\eta$ and the disturbance $w$ are Gaussian distributed.
We modify the reduction given by Dean et al. to the case when the excitation and disturbance are no longer Gaussian, but
instead zero-mean and bounded. We assume that $w_k$ and $\eta_k$ are both zero-mean sequences with independent coordinates and finite fourth moments. In particular, we assume $\E_{w_k}[w_k(i)^2] \leq \sigma_w^2$,
$\E_{w_k}[ w_k(i)^4 ] \lesssim \sigma_w^4$,
$\E_{\eta_k}[\eta_k(i)^2] \leq \sigma_\eta^2$,
$\E_{\eta_k}[ \eta_k(i)^4 ] \lesssim \sigma_\eta^4$.
These assumptions are quickly verified for common distributions such as
uniform on a compact interval or over a discrete set of points.
The main estimation result is the following.

\begin{theorem}
\label{thm:estimation}
Fix a {failure probability} $\delta \in (0, 1)$.
Suppose the stochastic disturbance $\{w_k\}$ and
the input disturbance $\{\eta_k\}$ satisfy the assumptions above.
Assume for simplicity that $\sigma_\eta \leq \sigma_w$, and that
the stabilizing controller $\tf K$ achieves a
SLS response $\tf \Phi_x \in \frac{1}{z} \RHinf(C_x, \rho), \tf \Phi_u \in \frac{1}{z} \RHinf(C_u, \rho)$.
Let $C_K^2 := n C_x^2 + d C_u^2$. Then as long as the trajectory length $T$ satisfies the condition:
\begin{align}
  T &\gtrsim T_0 := (n + d)\log\bigg( \frac{dC_u^2}{\delta}  + \frac{\sigma_w^2}{\sigma_\eta^2} \frac{\rho^2 C_u^2 C_K^2}{\delta(1-\rho^2)}
      \left(1 + \norm{\trueB}_2^2 + \frac{\norm{x_0}^2_2}{\sigma_w^2 T} \right)
  \bigg) \:,\label{eq:T0}
\end{align}
we have the following bound on the least-squares estimation errors that holds with
probability at least $1-\delta$,
\begin{align*}
  &\max\{ \norm{\Delta_A}_2, \norm{\Delta_B}_2 \} \lesssim \frac{\sigma_w C_u}{\sigma_\eta} \sqrt{\frac{n+d}{T}} \sqrt{ \log\left( \frac{dC_u}{\delta} + \frac{\sigma_w}{\sigma_\eta} \frac{\rho C_u C_K}{\delta(1-\rho^2)}
      \left(1 + \norm{\trueB}_2 + \frac{\norm{x_0}_2}{\sigma_w \sqrt{T}} \right)
\right) } \:.
\end{align*}
\end{theorem} \vspace{0.5em}
The proof of this result is presented in \extendedest.
We remark on the interpretation of statistical learning bounds. A priori guarantees, like the one presented here, depend on quantities related to the underlying true system. As such, they are not directly useful when the system is unknown,\footnote{
  Statistical bounds in terms of data-dependent quantities can also be worked out; however, modern methods like bootstrapping generally provide tighter statistical guarantees~\cite{efron1992bootstrap}.
} but rather they indicate qualities of systems that make them easier or harder to estimate.
{For example,
we see that this bound on estimation errors decreases with the size of the excitation, and increases with the size of the process noise. 
The bound also increases with $C_u$, a constant which bounds the gain from disturbance to control inputs.}

{Applying our robustness results to the learning controller~\eqref{eq:control_est}, it is possible to guarantee safety \emph{during} the estimation process.
We can define an expanded noise process $\tilde w_k = \trueB \eta_k + w_k$ to synthesize $\tf K_0$ using the robust control synthesis problem~\eqref{eq:sls_lqr_robust}. 
As long as it is feasible for initial system estimates $(A_0,B_0)$, initial dynamics uncertainties $(\eps_{A}^0,\eps_{B}^0)$, $\sigma_w$ replaced with\footnote{
  Note that since the quantity $\|\trueB\|_\infty$ would not generally be known, it can be bounded by $\|B_0\|_\infty+\eps_{B,\infty}^0$.
} $\sigma_\eta\|\trueB\|_\infty+\sigma_w$, and $b_{u,j}$ replaced with $b_{u,j}-\sigma_\eta\|F_{u,j}\|_1$, then the control law $\tf u = \tf K_0\tf x + \eta$ stabilizes the true system, satisfies state and input constraints, and allows for learning at the rate given in Theorem~\ref{thm:estimation}.
}

Finally, we connect the sub-optimality result to the statistical learning bound for an end-to-end sample complexity bound on the constrained LQR problem.
Define $M_T$ to be the value of $M_\zeta$ when the definition determined by~\eqref{eq:sls_lqr_rob_constraint} has values set as
\[\zeta_\infty = \eps\sqrt{n+d} \left\| \begin{bmatrix}\tf\Phi_x^\star \\ \tf\Phi_u^\star\end{bmatrix}\right\|_{\lone},~~ \zeta_2 =\eps \left\| \begin{bmatrix}\tf\Phi_x^\star \\\tf\Phi_u^\star\end{bmatrix}\right\|_{\hinf},~~\eps\gtrsim \frac{\sigma_w C_u}{\sigma_\eta} \sqrt{\frac{n+d}{T}}\calO(\sqrt{\log{(d/\delta)}}) \:. \]

\begin{coro}\label{coro:sample_complexity}
Assume initial feasibility of the learning problem. Then for a trajectory of length
\begin{align}T \gtrsim  T_0\frac{\sigma_w^2C_u^2}{\sigma_\eta^2}
&\max\Big\{(n+d) (1+\|\tf K_\star\|_{\lone})^2\|\tf\Phi_x^\star\|_{\lone}^2, (1+\|\tf K_\star\|_{\hinf})^2\|\tf\Phi_x^\star\|_{\hinf}^2\Big\}\:,\label{eq:Tcond}\end{align}
the cost achieved by $\Kh=\widehat{\tf \Phi}_u(\widehat{\tf\Phi}_x)^{-1}$ synthesized from~\eqref{eq:sls_lqr_robust} on the least-squares estimates $\Ah,\Bh$  satisfies with probability at least $1-\delta$,
\begin{align*}
&\frac{J(\trueA,\trueB,\widehat{\tf K}) - J_\star}{J_\star} \lesssim 
\frac{\sigma_w C_u}{\sigma_\eta} \sqrt{\frac{n+d}{T}}  (1+M_T)(1+\|\tf K_\star \|_{\hinf})\|\tf\Phi^\star_x\|_{\hinf} \calO(\sqrt{\log{(d/\delta)}})+M_T
\:.
\end{align*}
\end{coro}
\begin{proof}(Sketch)
This result follows by combining the statistical guarantee in Theorem~\ref{thm:estimation} with the sub-optimality bound in Theorem~\ref{thm:suboptimality}. Note that we use the na{\"i}ve bound $\eps_{A,\infty}\leq \sqrt{n}\eps_{A,2}$ and similarly $\eps_{B,\infty} \leq \sqrt{d} \eps_{B,2}$; this results in
an extra factor of $(n+d)$ appearing in \eqref{eq:Tcond} and in the definition of $M_T$.
\end{proof}

Our final result depends both on the true system and the initial system estimates by way of the learning controller, which affects $T_0$ and constants $C_x$, $C_u$, and $\rho$. The system constraints enter through their effect on  $M_T$.

\section{Numerical Experiments} \label{sec:experiments}

We demonstrate the utility of this framework on a double integrator example. In this case, the true dynamics are given by
\[ x_{k+1} = \begin{bmatrix} 1 & 0.1 \\ 0 & 1 \end{bmatrix} x_k + \begin{bmatrix} 0\\ 1 \end{bmatrix} u_k + w_k \]
with the constraints as states bounded between $-8$ and $8$, and inputs bounded in between $-4$ and $4$. We have $\sigma_w=0.1$. Our initial estimate comes from a randomly generated initial perturbation of the true system with $\eps_{A,\infty}=\eps_{B,\infty}=0.1$. Safe controllers are generated with finite truncation length $L=15$, and for larger initial conditions, the system is warm-started with a finite-time robust controller with horizon $20$ to reduce the initial condition.

\begin{figure}
\centering
\includegraphics[width=\imagewidth]{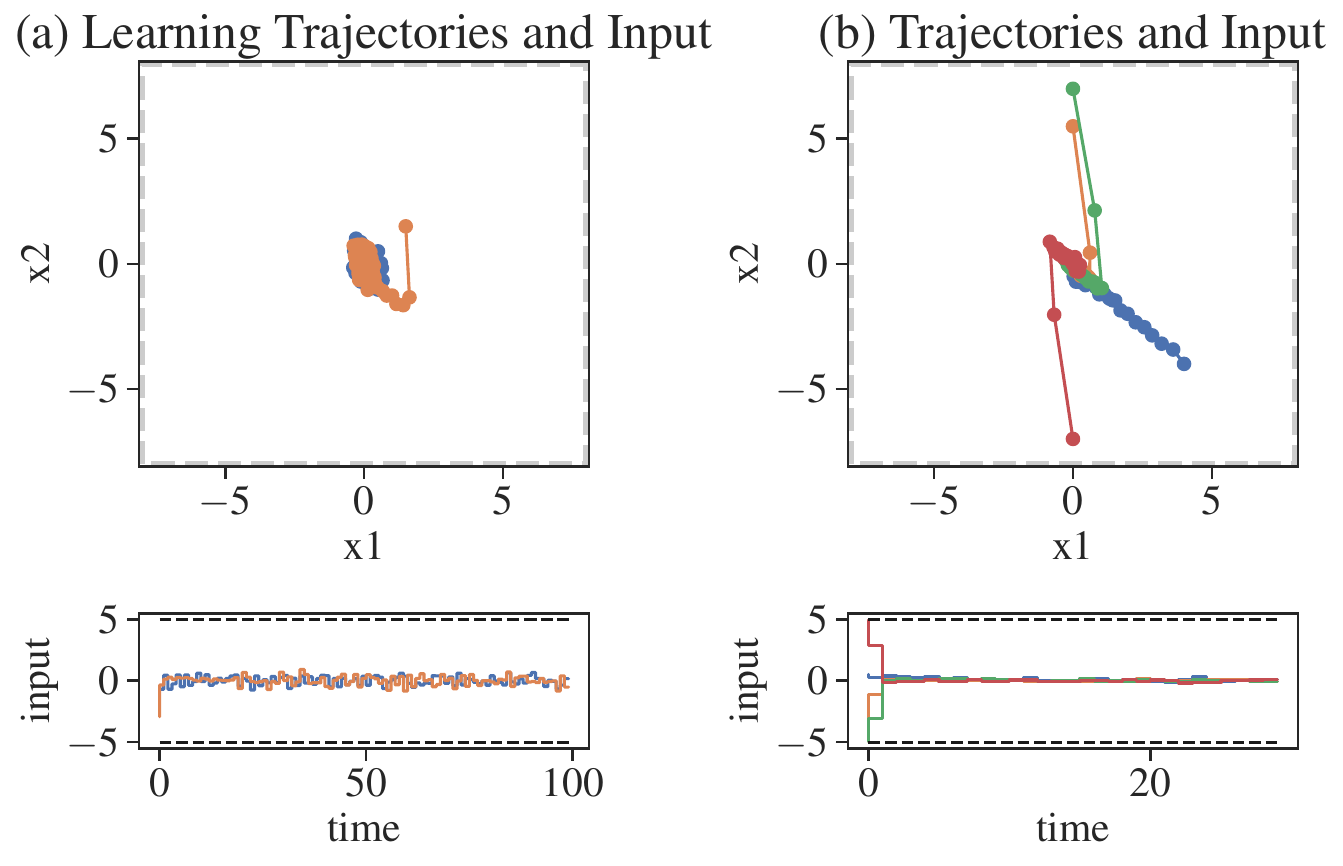}
\caption{Safe learning trajectories synthesized with coarse initial estimates (a), then robust execution with reduced model errors (b).}\label{fig:learning}
\end{figure}

Figure~\ref{fig:learning} displays safe trajectories and input sequences for several example initial conditions. In~\ref{fig:learning}a, the plotted trajectories are used for learning: the controller both regulates and excites the system ($\sigma_\eta=0.5$), and is robust to initial uncertainties.
Figure~\ref{fig:learning}b demonstrates an ability to operate closer to the constraints when there is less uncertainty: in this case, there is no added excitation ($\sigma_\eta=0$) and the system estimates are better specified ($\eps_{\infty}=0.001$), so larger initial conditions are feasible.

\begin{figure}
\centering
\includegraphics[width=\textwidth]{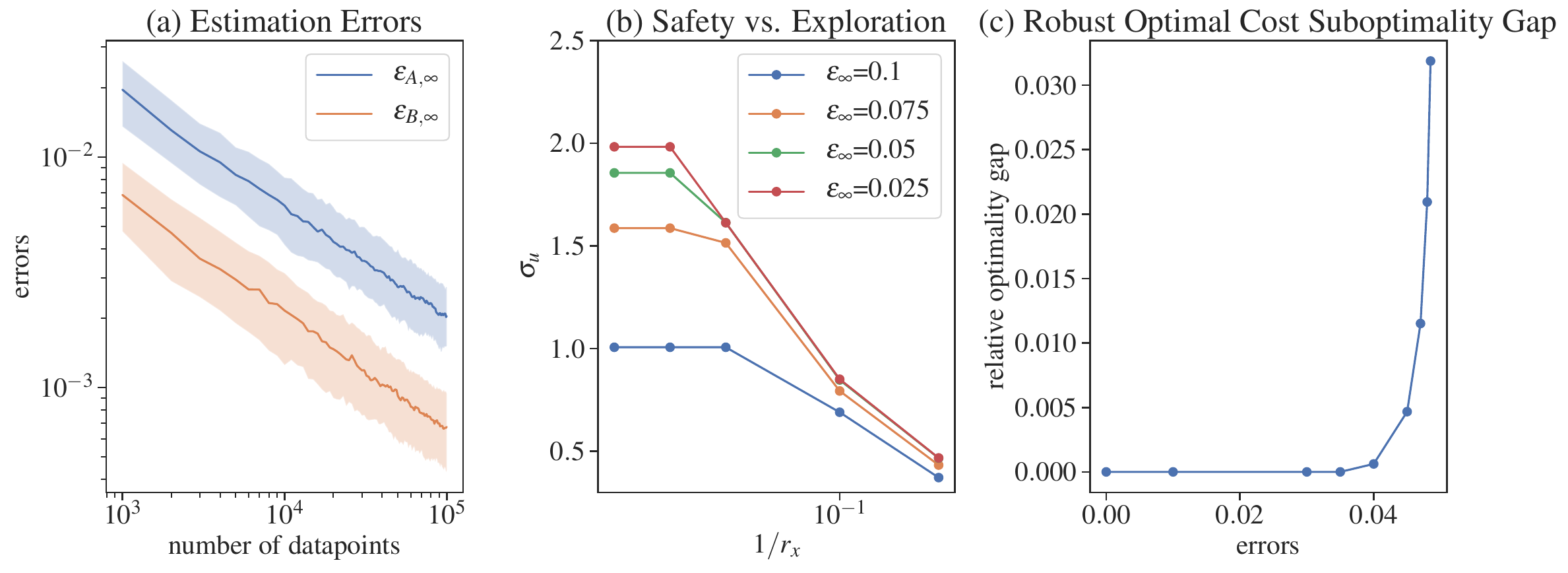}
\caption{Over time, estimation errors decrease (a). As safety requirements increase, the maximum feasible excitation decreases (b). The robust cost suboptimality gap $M_{\zeta}$ displays an abrupt transition from feasibility to near-optimality with error size (c).}\label{fig:tradeoff}
\end{figure}
Figure~\ref{fig:tradeoff}a displays the decreasing estimation errors over time, demonstrating learning. Shaded areas represent quartiles over $400$ trials. 
Figure~\ref{fig:tradeoff}b displays the trade-off between safety and exploration by showing the largest value of $\sigma_\eta$ for which the robust synthesis is feasible, given a size for the state constraint set $r_x$. Here, we leave $x_0=0$, and examine a variety of errors in the dynamics estimates. 
As the uncertainties in the dynamics decrease, higher levels of both safety and exploration are achievable. 
Finally, Figure~\ref{fig:tradeoff}c shows the value of the relative robust sub-optimality gap $M_{\zeta}$ for the given example. We see a sharp transition from infeasibility for $\eps\geq0.05$ to near-optimality for $\eps\leq0.03$. 
This indicates that the gap may be most significant as a feasibility condition for our sub-optimality guarantees to hold.

\section{Discussion}
In this paper, we propose a method for learning unknown linear systems while ensuring that they satisfy state and input constraints. By synthesizing a controller that both excites and regulates the system, we address the trade-off between safety and exploration directly. We further derive an end-to-end finite sample bound on the performance of constrained LQR controllers synthesized from collected data.

There are several directions for possible extensions of this work. 
To mitigate the conservativeness of the robust controller, tighter bounds on the uncertainty in the system response $\Dh$ could be derived for structured settings, where more than just the norm of the error is known.
To connect this work to experiment design literature, the objective in the synthesis problem~\eqref{eq:sls_lqr_robust} could be replaced with an exploration inspired cost function for the learning stage.

Alternatively, the constrained LQR problem could be cast in the setting of online learning, where one seeks to minimize cost at all times, including during learning. This would require an analysis of recursive feasibility, to understand the transition that occurs when controllers are updated based on refined system estimates. 
It would also likely require a finer analysis on the performance loss characterized by $M_\zeta$.
Finally, we remark that the exploration vs. safety trade-off is most compelling for nonlinear systems, and a linear-time-varying extension to this work may be applicable to those problems.


\section*{ACKNOWLEDGMENT}

We thank Francesco Borrelli and the members of the MPC Lab at UC Berkeley for their helpful comments and feedback.
SD is supported by an NSF Graduate Research Fellowship under Grant No. DGE 1752814.
ST is supported by a Google PhD fellowship.
BR is generously supported in part by ONR awards N00014-17-1-2191, N00014-17-1-2401, and N00014-18-1-2833, the DARPA Assured Autonomy (FA8750-18-C-0101) and Lagrange (W911NF-16-1-0552) programs, and an Amazon AWS AI Research Award.


\bibliographystyle{IEEEtran}
\bibliography{refs}

\newpage
\appendix
\section{Extended Sub-optimality Results and Proofs}\label{sec:all_sls_proofs}

This section contains proofs necessary for the sub-optimality results. It further develops a sub-optimality bound for the case of FIR truncation. We begin by considering the frequency response elements of composed transfer functions. First, define, for $\tf D =\sum_{k=1}^\infty D(k)z^{-k}$,
\[\Toep_k(D) := \begin{bmatrix}
D(1) & & \\
\vdots & \ddots & &\\
D(k) &  \dots & D(1)
  \end{bmatrix} \:.\]
Then we have the following lemma.

\begin{lemma} \label{lem:freq_resp_composition}
For $\tf M = \sum_{k=1}^\infty M(k)z^{-k}$ and $\tf D =\sum_{k=1}^\infty D(k)z^{-k}$ the frequency response elements of $\tilde{\tf M} = \tf M \tf D$ are given by
\[\tilde M(k)= M[k:1]D(1:k)\:.\]
where we use the notation $D(1:k)$ for the vertical concatenation of $D(1)$ through $D(k)$.
Further, we have
\begin{align*}
\tilde M[k:1] &= M[k:1] {\Toep_k(D)}  \:.
\end{align*}
\end{lemma}
\begin{proof}
For the first statement, notice that we can write
\[\tf M \tf D = \sum_{k=1}^\infty\sum_{t=1}^{k} M(t)D(k-t)z^{-k} = \sum_{k=1}^\infty M[k:1]D(1:k)z^{-k}\:.\]
Then for the second, we have
\begin{align*}
\tilde M[k:1] =& \begin{bmatrix} M[k:1]D(1:k) &\dots & M[1]D(1) \end{bmatrix}\\ &= M[k:1] \underbrace{\begin{bmatrix}
D(1) & & \\
\vdots & \ddots & &\\
D(k) &  \dots & D(1)
  \end{bmatrix}}_{\Toep_k(D)}  \:.
\end{align*}
\end{proof}

\subsection{Proof of Lemma~\ref{lem:feasible_soln_bound}}\label{sec:sls_proofs}

\begin{proof}[Proof of Lemma~\ref{lem:feasible_soln_bound}]
The proposed feasible solution is
\begin{align*}
\tf{\tilde{\Phi}}_x = \tf\Phi^c _x(I-\tf\Delta)^{-1},~ \tf{\tilde{\Phi}}_u = \tf\Phi^c_u(I-\tf\Delta)^{-1},~
~\tilde\gamma=\frac{\sqrt{2}\zeta_2}{1-\sqrt{2}\zeta_2},~~\tilde\tau=\frac{\zeta_\infty}{1-\zeta_\infty}\:,
\end{align*}
where ${\tf\Delta} = -\begin{bmatrix} \Delta_A & \Delta_B \end{bmatrix} \begin{bmatrix} \tf\Phi_x^c \\ \tf\Phi_u^c \end{bmatrix}$.
By construction, $\Phixtilde$ and $\Phiutilde$ satisfy the equality constraints. 
Next, notice that 
\[\|\tf\Delta\|_{\hinf} = \left\|\begin{bmatrix} \frac{1}{\eps_{A,2}}\Delta_A & \frac{1}{\eps_{B,2}}\Delta_B \end{bmatrix} \begin{bmatrix} \eps_{A,2}\tf\Phi_x^c \\ \eps_{B,2}\tf\Phi_u^c \end{bmatrix}\right\|_{\hinf} \leq \sqrt{2}\zeta_2\:,\]
due to the sub-multiplicative property of the $\hinf$ norm and the norm constraint in the definition of $(\tf\Phi^c_x,\tf\Phi^c_u)$. 
By similar logic $\|\tf\Delta\|_{\lone}\leq \zeta_\infty$. 

Considering the $\hinf$ norm constraint,
\begin{align*}
  \sqrt{2} \left\| \begin{bmatrix}\eps_{A,2} \Phixtilde \\ \eps_{B,2}\Phiutilde\end{bmatrix}\right\|_{\hinf} & = \sqrt{2}\left\| \begin{bmatrix}\eps_{A,2} \tf\Phi^c_x \\ \eps_{B,2}\tf\Phi^c_u\end{bmatrix}(I-\tf\Delta)^{-1}\right\|_{\hinf} \leq 
  \sqrt{2}\zeta_2 \frac{1}{1-\|\tf\Delta\|_{\hinf}} \leq  \frac{\sqrt{2}\zeta_2}{1-\sqrt{2}\zeta_2} =\tilde \gamma\:,
  \end{align*}
  where the inequality follows due to the fact that $\sqrt{2}\|\tf\Delta\|_{\hinf}\leq 1$.
 Then consider the $\lone$ norm constraint,
\begin{align*}
  \left\| \begin{bmatrix}\eps_{A,\infty}\Phixtilde \\ \eps_{B,\infty}\Phiutilde\end{bmatrix}\right\|_{\lone} &\leq 
  \zeta_\infty \frac{1}{1-\|\tf\Delta\|_{\lone}}  \leq \frac{\zeta_\infty }{1-\zeta_\infty} = \tilde\tau
\end{align*}
where the decomposition of the inverse is valid by assumption on the size of $\zeta_\infty$.

Then it remains to show that the robust state and input constraints are satisfied. For compactness we will write $c_0=\max(1, \frac{1}{\sigma_w}\|x_0\|_\infty)$. Recall that they are given by
\[G_x^\tau(\tilde{\tf\Phi}_x;k)_j = F_{x,j}^\top \tilde\Phi_x(k+1)x_0 +  \sigma_w \|F_{x,j}^\top \tilde\Phi_x[k:1]\|_1\:
+\frac{\tau \sigma_wc_{0}}{1-\tau}\|F_{x,j}^\top\tilde\Phi_x[k+1:1]\|_1\:.\]

Note that $\tilde{\tf\Phi}_x=\Phixc + \Phixc\tf\Delta(I-\tf\Delta)^{-1}$.
Define the frequency response elements of $\tf\Delta(I-\tf\Delta)^{-1}$ by $D(k)$. Then, using Lemma~\ref{lem:freq_resp_composition} we have
\[\tilde\Phi(k)= \Phi_x^c(k) + \Phi^c_x[k:1]D(1:k)~,~~ \tilde\Phi[k:1] = \Phi^c_x[k:1] + \Phi^c_x[k:1] {\Toep_k(D)}\:.\]

\begin{lemma} \label{lem:toep_bound}
Assume that $\zeta_\infty<1$ and let $D$ represent the frequency response of $\tf\Delta(1-\tf\Delta)^{-1}$.
Then for any vector $v$,
\[\|v^\top \Toep_k(D)\|_1 \leq \frac{\zeta_\infty}{1-\zeta_\infty}\|v\|_1\:.\]
\end{lemma}
\begin{proof}
\begin{align*}
\|v^\top \Toep_k(D)\|_1 &= \|\Toep_k(D)^\top v\|_1 \leq \|\Toep_k(D)^\top\|_{1} \|v\|_1= \|\Toep_k(D)\|_{\infty} \|v\|_1\\
& \leq \|\tf\Delta(1-\tf\Delta)^{-1}\|_{\lone}\|v\|_1 \leq \frac{\zeta_\infty}{1-\zeta_\infty}\|v\|_1\:.
\end{align*}
Above, we make use of the fact that the $\ell_1$ and $\ell_\infty$ norms are duals, and therefore $\|A\|_{\infty}=\|A^\top\|_{1}$.
The second inequality holds because $\Toep_k(D)$ is a truncation of the semi-infinite Toeplitz matrix associated with the operator $\tf\Delta(1-\tf\Delta)^{-1}$. The final decomposition is valid because $\|\tf\Delta\|_{\lone} \leq \zeta_\infty <1$.
\end{proof}
Now we are ready to consider the state constraint indexed by $j$ and $k$,
\begin{align*}
G_x^{\tilde\tau}(\tilde{\tf\Phi}_x;k)_j
&= F_{x,j}^\top  (\Phi_x^c(k+1) + \Phi^c_x[k+1:1]D(1:k+1)) x_0 \\
&+  \sigma_w \|F_{x,j}^\top \Phi^c_x[k:1] (I+ \Toep_k(D) )\|_1\\
&+\frac{\tilde\tau}{1-\tilde\tau}\sigma_wc_{0}\|F_{x,j}^\top \Phi^c_x[k+1:1](I+ \Toep_{k+1}(D))\|_1
\end{align*}
Considering each term individually,
\begin{align*}
F_{x,j}^\top  (\Phi_x^c(k+1) &+ \Phi^c_x[k+1:1]D(1:k+1)) x_0\\
 &= F_{x,j}^\top \Phi_x^c(k+1) x_0 + F_{x,j}^\top \Phi^c_x[k+1:1]D(1:k+1) x_0  \:,
 \end{align*}
Then defining $E_1$ to contain an identity in the first block and zeros elsewhere,
\begin{align*}
 F_{x,j}^\top \Phi^c_x[k+1:1]D(1:k+1) x_0 &\leq
  \|F_{x,j}^\top \Phi^c_x[k+1:1]\Toep_{k+1}(D)E_1 \|_1 \|x_0\|_\infty  \\
&\leq  \|F_{x,j}^\top \Phi^c_x[k+1:1]\|_1 \frac{\zeta_\infty}{1-\zeta_\infty} \|x_0\|_\infty\:,
\end{align*}
The first inequality is H{\"o}lder's inequality and the second by Lemma~\ref{lem:toep_bound}. Next, the second term:
\begin{align*}
\sigma_w \|F_{x,j}^\top \Phi^c_x[k:1] (I+ \Toep_k(D) )\|_1&\leq \sigma_w  \frac{1}{1-\zeta_\infty} \|F_{x,j}^\top \Phi^c_x[k:1]\|_1\:,
\end{align*}

where the inequality holds by Lemma~\ref{lem:toep_bound}. Finally, the last term,
\begin{align*}
\frac{\tilde\tau}{1-\tilde\tau}\sigma_wc_{0}\|F_{x,j}^\top \Phi^c_x[k+1:1]&(I+ \Toep_{k+1}(D))\|_1 \leq
\frac{\zeta_\infty}{1-2\zeta_\infty}\sigma_wc_{0} \frac{1}{1-\zeta_\infty}
\|F_{x,j}^\top \Phi^c_x[k+1:1]\|_1
\end{align*}
where we apply Lemma~\ref{lem:toep_bound} and plug in the definition of $\tilde \tau$:
\begin{align*}
 \frac{\tilde\tau}{1-\tilde\tau} &=
\frac{\zeta_\infty}{(1-\zeta_\infty)(1-\frac{\zeta_\infty}{1-\zeta_\infty})} = \frac{\zeta_\infty}{1-2\zeta_\infty}\:.
\end{align*}
The resulting sum is

\begin{align*}
G_x^{\tilde\tau}(\tilde{\tf\Phi}_x;k)_j
&\leq F_{x,j}^\top \Phi_x^c(k+1) x_0  + 
 \frac{\sigma_w}{1-\zeta_\infty} \|F_{x,j}^\top \Phi^c_x[k:1]\|_1\\
& + \left(\frac{\zeta_\infty}{1-\zeta_\infty} \|x_0\|_\infty +  \frac{\zeta_\infty}{1-2\zeta_\infty} \frac{ c_{0}  \sigma_w}{1-\zeta_\infty}\right)
\|F_{x,j}^\top \Phi^c_x[k+1:1]\|_1
\end{align*}
Then considering constants around the final term,
\begin{align*}
 \frac{\zeta_\infty}{1-\zeta_\infty} \|x_0\|_\infty + \frac{\zeta_\infty}{1-2\zeta_\infty}\frac{ c_{0}\sigma_w}{1-\zeta_\infty}
\leq\frac{ 1}{1-\zeta_\infty}\left(
\zeta_\infty + \frac{\zeta_\infty}{1-2\zeta_\infty}\right)  c_{0}\sigma_w
=\frac{2\zeta_\infty}{1-2\zeta_\infty}  c_{0}\sigma_w\:.
\end{align*}
Thus, we see that
$G_x^{\tilde\tau}(\tilde{\tf\Phi}_x;k)_j
\leq \bar G_x^{\zeta}(\tf\Phi^c_x;k)_j \leq b_j$
due to the constraints on $\tf\Phi^c_x$. A similar computation with the input constraints shows the same result. Therefore, the proposed solution is feasible.
\end{proof}

\subsection{Finite Dimensional Sub-optimality}
\newcommand{\Dhat}{\widehat{\tf{\Delta}}}
We can recover sub-optimality guarantees in the case that we optimize over only a finite set of system response variables. Define the truncated responses $\tf\Phi_x^L = \sum_{k=1}^L \Phi_x(t) z^{-k}$ and similarly for $\tf\Phi_u^L$, and notice that
\begin{align}\label{eq:FIR_tail}
\begin{bmatrix} zI -A & -B \end{bmatrix}  \begin{bmatrix}\tf\Phi_x \\ \tf\Phi_u\end{bmatrix} = I   \iff \begin{bmatrix} zI -A & -B \end{bmatrix}  \begin{bmatrix}\tf\Phi^L_x \\ \tf\Phi^L_u\end{bmatrix} = I + \frac{1}{z^L} \Phi_x(L+1) \:. 
\end{align}
This reformulation allows for the optimization over the FIR filters $\tf\Phi_x^L$ and $\tf\Phi_u^L$, plus an additional variable that represents the tail of the true response.
{Applying this observation to the robust synthesis problem \eqref{eq:sls_lqr_robust} yields the following} finite dimensional problem:
\begin{align}
\widehat J_L(\gamma,\tau):=\min_{
\substack{\tf\Phi_x^L,\tf\Phi_u^L\in\RHinf\\V}} ~ &\frac{1}{1-\gamma} J(\Ah, \Bh, \tf K) \label{eq:sls_lqr_robust_FIR} \\
\mathrm{s.t.}~& \begin{bmatrix} zI - \Ahat & -\Bhat \end{bmatrix} \begin{bmatrix}\tf\Phi_x^L \\ \tf\Phi_u^L\end{bmatrix} = I + \frac{1}{z^L}V,\notag \\
  & \sqrt{2} \left\| \begin{bmatrix}\eps_{A,2}\tf\Phi_x^L \\ \eps_{B,2}\tf\Phi_u^L\end{bmatrix}\right\|_{\hinf} + \left\| V \right\|_2 \leq \gamma, ~~ \left\| \begin{bmatrix}\eps_{A,\infty}\tf\Phi_x^L \\ \eps_{B,\infty}\tf\Phi_u^L\end{bmatrix}\right\|_{\lone} + \|V\|_\infty \leq \tau, \notag\\
&G_x^\tau(\tf\Phi_x^L;k) \leq b_x, ~~ G_u^\tau(\tf\Phi_u^L;k) \leq b_u~~\forall ~k\:. \notag
\end{align}

{We now show that so long as the horizon $L$ is suitably large, all of the properties that hold for the solution of infinite horizon problem \eqref{eq:sls_lqr_robust} carry over to the solution of the finite dimensional approximation \eqref{eq:sls_lqr_robust_FIR}.}

\begin{proposition}\label{prop:robust_constraint_satisfaction_FIR}
Any controller designed from a feasible solution to the finite robust control problem~\eqref{eq:sls_lqr_robust_FIR} for any $0\leq \gamma,\tau<1$ stabilizes the true system and ensures that state and input constraints will be satisfied.
\end{proposition}

\begin{proof}
Consider the affine constraint on $\tf\Phi_x^L$ and $\tf\Phi_u^L$. We can equivalently write, using the observation in~\eqref{eq:FIR_tail},
\begin{align*}
\begin{bmatrix} zI - \Ahat & -\Bhat \end{bmatrix} \begin{bmatrix}{\tf\Phi}^L_x \\ {\tf\Phi}_u^L\end{bmatrix} &= I + \frac{1}{z^L}V
\implies 
\begin{bmatrix} zI - \trueA & -\trueB \end{bmatrix} \begin{bmatrix}{\tf\Phi}^L_x \\ {\tf\Phi}_u^L\end{bmatrix} = I + \underbrace{\frac{1}{z^L}V + \widehat{\tf\Delta}}_{\Dh_L}
\end{align*}
where we define $\widehat{\tf\Delta} = \Delta_A \tf\Phi_x^L + \Delta_B \tf\Phi_u^L$. The noting that
\begin{align*}
\begin{split}
&\|\Dh_L\|_{\hinf} \leq \sqrt{2} \left\| \begin{bmatrix}\eps_{A,2}\tf\Phi_x^L \\ \eps_{B,2}\tf\Phi_u^L\end{bmatrix}\right\|_{\hinf} + \|V\|_2~,\\ &\|\Dh_L\|_{\lone} \leq \left\| \begin{bmatrix}\eps_{A,\infty}\tf\Phi_x^L \\ \eps_{B,\infty}\tf\Phi_u^L\end{bmatrix}\right\|_{\lone}+\|V\|_\infty\:,
\end{split}
\end{align*}
by~\eqref{eq:delta_perf_bound},
the system is stabilized and the true system trajectory is given by
\[\tf x = \tf \Phi_x (I+\Dh_L)^{-1} \tf w,\quad \tf u = \tf \Phi_u (I+\Dh_L)^{-1} \tf w\:.\]
The rest of the proof follows exactly as in the proof of Theorem~\ref{thm:robust_constraint_satisfaction}
\end{proof}

To bound the sub-optimality of this robust synthesis, we need to connect the optimal controller to the robust problem. 
For this feasibility result, it is necessary to understand the system response induced by putting the optimal robustly constrained controller $\tf K_c=\Phiuc(\Phixc)^{-1}$, as defined in~\eqref{eq:sls_lqr_rob_constraint}, in feedback with the estimated dynamics $(\Ahat, \Bhat)$. We denote this system responses as 
$\Phixc(I-\tf\Delta)^{-1}$.
As before, ${\tf\Delta} = -\Delta_A\tf\Phi_x^c - \Delta_B \tf\Phi_u^c $. 

We thus define constants related to the decay rate of the system responses. Let $C_\star$ and $\rho_\star$ be any constants such that $\|\Phi_x^c(k)\|_p\leq C_\star\rho_\star^k$ for $p\in\{2,\infty\}$. Note that these constants must exist since the optimal closed loop will be exponentially stable. Let $C_\Delta$ and $\rho_\Delta$ similarly be any constants that bound the norm of the frequency response elements of $(I-\tf\Delta)^{-1}$. We note that these constants can be bounded by multiples of $C_\star$ and $\rho_\star$, as worked out in Appendix G of Dean et al.~\cite{dean2018regret}.  

\begin{lemma}\label{lem:C_rho}
Suppose that for $p\in\{2,\infty\}$, $\|\Phi_x^c(k)\|_p\leq C_\star\rho_\star^k$ and $\|(I-\tf\Delta)^{-1}[k]\|_p\leq C_\Delta\rho_\Delta^k$.
Define $C=\frac{C_\star C_{\Delta}}{\log (\frac{\max(\rho_\star,\rho_\Delta)+1}{2\max(\rho_\star,\rho_\Delta)})}$ and $\rho=\frac{\max(\rho_\star,\rho_\Delta) + 1}{2}$.
Then for $p\in\{2,\infty\}$ and all $t\geq 0$,
\[\|\Phixc(I-\tf\Delta)^{-1}[k]\|_p\leq 
C
\rho^k \:.\]
\end{lemma}
We defer the proof of this lemma to the end of the section. 
Next, we redefine the optimal robustly constrained responses
$(\tf\Phi^c_x, \tf\Phi^c_u)$ in~\eqref{eq:sls_lqr_rob_constraint}
to satisfy the modified doubly robust constraint sets
\[\bar G_x^{\zeta}(\tf\Phi_x;k)_j =  
{F_{x,j}^\top \Phi(k+1)x_0 + \frac{\sigma_w}{1-\zeta_\infty} \|F_{x,j}^\top \Phi_x[k:1]\|_1}
 + {\frac{3\zeta_\infty}{1-3\zeta_\infty} \sigma_w c_{0}\|F_{x,j}^\top \Phi_x[k+1:1]\|_1}\:,\]
 Here, we have changed a factor of $2$ to a factor of $3$ in the last term to account for the finite approximation. We will define $M_{\zeta}$ as in~\eqref{eq:gap} using the new $(\tf\Phi^c_x, \tf\Phi^c_u)$.
We are now ready to state the sub-optimality result.

\begin{theorem}\label{thm:suboptimalityFIR} 
Suppose that the truncation length is such that
\[
  L \geq \frac{\log\left( C\max(\frac{1-\sqrt{2}\zeta_2}{\sqrt{2}\zeta_2},\frac{1-\zeta_\infty}{\zeta_\infty})\right)}{\log(1/\rho)} - 1 \:.
\]
As long as
$\zeta_\infty <  \frac{1}{3}$ and $\zeta_2 < \frac{1}{3\sqrt{2}}$,
we have that the cost achieved by $\Kh(L)=\widehat{\tf \Phi}_u\widehat{\tf\Phi}_x^{-1}$ synthesized from the minimizers of $\min_{\gamma,\tau}~\widehat J_L(\gamma,\tau)$
satisfies
  \[\frac{J(\trueA,\trueB,\widehat{\tf K}_L) - J_\star}{J_\star} \leq 6\sqrt{2}(1+M_{\zeta})(\eps_{A,2}\|\tf\Phi^\star_x\|_{\hinf}+\eps_{B,2}\|\tf\Phi^\star_u\|_{\hinf})+M_{\zeta} \:.\]
\end{theorem}

To prove this result, we first
construct a feasible solution to the truncated synthesis problem~\eqref{eq:sls_lqr_robust_FIR}.

\begin{lemma}\label{lem:FIR_feasibility}
Under the conditions of Theorem~\ref{thm:suboptimalityFIR},
 the following is a feasible solution to~\eqref{eq:sls_lqr_robust_FIR}:
\begin{align*}
\tilde{\tf\Phi}_x = \tf\Phi^c_x(I-\tf\Delta)^{-1}[1:L],&~~\tilde{\tf\Phi}_u = \tf\Phi^c_u(I-\tf\Delta)^{-1}[1:L],~~\tilde V=\tf\Phi^c_x(I-\tf\Delta)^{-1}[L+1],\\
&~~\tilde\gamma =  \frac{2\sqrt{2}\zeta_2}{1-\sqrt{2}\zeta_2},~~ \tilde\tau = \frac{2\zeta_\infty}{1-\zeta_\infty}.
\end{align*}
\end{lemma}

\begin{proof}[Proof of Lemma~\ref{lem:FIR_feasibility}]
First, we consider the affine constraint,
\begin{align*}
\begin{bmatrix} zI - \Ahat & -\Bhat \end{bmatrix} \begin{bmatrix}\tilde{\tf\Phi}_x \\ \tilde{\tf\Phi}_u\end{bmatrix} = I + \frac{1}{z^L}V\:.
\end{align*}
Applying the observation in~\eqref{eq:FIR_tail},

\begin{align*}
\begin{bmatrix} zI - \Ahat & -\Bhat \end{bmatrix} \begin{bmatrix}\tilde{\tf\Phi}_x \\ \tilde{\tf\Phi}_u\end{bmatrix} &= I + \frac{1}{z^L}V
\iff 
\begin{bmatrix} zI - \Ahat & -\Bhat \end{bmatrix} \begin{bmatrix} \Phixc \\ \Phiuc \end{bmatrix} (I-\tf\Delta)^{-1} = I\\
&\iff 
\Big(\begin{bmatrix} zI - \trueA & -\trueB \end{bmatrix} + \begin{bmatrix} \Delta_A & \Delta_B \end{bmatrix}\Big) \begin{bmatrix} \Phixc \\ \Phiuc \end{bmatrix} (I-\tf\Delta)^{-1} = I
\end{align*}
The right hand side of this equation is true by the definition of $\tf\Delta$ and the affine constraint on $\Phixc$ and $\Phiuc$, and thus the affine constraint on $\tilde{\tf\Phi}_x$ and $\tilde{\tf\Phi}_u$ is satisfied.

Next, we bound the norm of $\tilde V$. Using Lemma~\ref{lem:C_rho}, we have that
$$\|\Phixc(I-\tf\Delta)^{-1}[L+1]\|_p\leq C\rho^{L+1}\:.$$

Next, we consider the $\hinf$ norm constraint
\begin{align*}
\sqrt{2} \left\| \begin{bmatrix}\eps_{A,2}\Phixtilde \\ \eps_{B,2}\Phiutilde\end{bmatrix}\right\|_{\hinf} + \|V\|_2 &
\leq\sqrt{2} \left\| \begin{bmatrix}\eps_{A,2}\Phixc\\ \eps_{B,2}\Phiuc\end{bmatrix}(I-\tf\Delta)^{-1}\right\|_{\hinf} + C\rho^{L+1}\\
&\leq \sqrt{2} \zeta_2\frac{1}{1-\|\tf\Delta\|_{\hinf}} + \frac{\sqrt{2} \zeta_2}{1-\sqrt{2}\zeta_2} \\
&\leq  \frac{2\sqrt{2} \zeta_2}{1-\sqrt{2}\zeta_2} = \tilde\gamma  \:.
\end{align*}
The first inequality follows from the bound on $\|\tilde V\|_2$ and because the proposed $\tilde{\tf \Phi}_x$ is a truncation of $\Phixc(I-\tf\Delta)^{-1}$ and similarly for $\tilde{\tf \Phi}_u$. The second follows by the sub-multiplicative property, because $\|\tf\Delta\|_{\hinf}\leq\zeta_2<1$, and due to the assumption on $L$.
Then similarly considering the $\lone$ norm constraint,
\begin{align*}
  \left\| \begin{bmatrix}\eps_{A,\infty}\Phixtilde \\ \eps_{B,\infty}\Phiutilde\end{bmatrix}\right\|_{\lone} + \|V\|_\infty
  &\leq \left\| \begin{bmatrix}\eps_{A,2}\Phixc\\ \eps_{B,2}\Phiuc\end{bmatrix}(I-\tf\Delta)^{-1}\right\|_{\lone} + C\rho^{L+1}\\
&\leq \zeta_\infty \frac{1}{1-\|\tf\Delta\|_{\lone}}+ \frac{\zeta_\infty }{1-\zeta_\infty}\\
&\leq \frac{2\zeta_\infty }{1-\zeta_\infty}= \tilde\tau \:.
\end{align*}
Then it remains only to show that the robust state and input constraints are satisfied.  
Notice that $\tilde{\tf\Phi}_x$ is a truncation of ${\Phixc}(1-\tf\Delta)^{-1}$. For $k>L$, the frequency response elements are zero, and therefore the constraints are trivially satisfied.
For for $0\leq k\leq L$, we have, by Lemma~\ref{lem:freq_resp_composition}, {using block-Toeplitz notation, that}
\[\tilde\Phi(k)= \Phi_x^c(k) + \Phi^c_x[k:1]D(1:k)\]
where we take $D(k)$ to represent the frequency response elements of $\tf\Delta(1-\tf\Delta)^{-1}$.
Then the proof of constraint satisfaction follows as in Section~\ref{sec:sls_proofs}, up to the computations with the definition of $\tilde \tau$
\begin{align*}
 \frac{\tilde\tau}{1-\tilde\tau} &=
\frac{2\zeta_\infty}{(1-\zeta_\infty)(1-\frac{\zeta_\infty}{1-\zeta_\infty})} = 
\frac{2\zeta_\infty}{1-3\zeta_\infty} 
\:.
\end{align*}
This affects the constant factors around the final term of the sum
\begin{align*}
\frac{\zeta_\infty}{1-\zeta_\infty} \|x_0\|_\infty +  
\frac{2\zeta_\infty}{1-3\zeta_\infty}  \frac{ c_{0}  \sigma_w}{1-\zeta_\infty} \leq 
\frac{1}{1-\zeta_\infty} \left( \zeta_\infty +  
\frac{2\zeta_\infty}{1-3\zeta_\infty} \right)  c_{0}  \sigma_w = \frac{3\zeta_\infty}{1-3\zeta_\infty} c_{0}  \sigma_w
\end{align*}
Thus, we see that
$G_x^{\tilde\tau}(\tilde{\tf\Phi}_x;k)_j
\leq \bar G_x^{\zeta}(\tf\Phi^c_x;k)_j \leq b_j$
due to the constraints on $\tf\Phi^c_x$. 
A similar computation with the input constraints shows the same result. Therefore, the proposed solution is feasible.
\end{proof}

We are now ready to prove the main sub-optimality result.
\begin{proof}[Proof of Theorem~\ref{thm:suboptimalityFIR}]
Recall that we denote the minimizers of $\min_{\gamma,\tau}~\eqref{eq:sls_lqr_robust_FIR}$ as $(\Phixh, \Phiuh, \widehat V, \widehat\gamma, \widehat\tau)$.
Let
\[
\Dh_L :=  \Dh+\frac{1}{z^L}\widehat V.
\]
Then we have that that
\[
  \hinfnorm{\Dh_L} \leq \sqrt{2} \bighinfnorm{\begin{bmatrix}\eps_{A,2} \Phixh \\ \eps_{B,2} \Phiuh \end{bmatrix}} + \|\widehat V\|_2 \leq \widehat\gamma,
\]
By the constraints of the optimization problem in~\eqref{eq:sls_lqr_robust_FIR}.

We now apply observation~\ref{eq:FIR_tail} and~\eqref{eq:delta_perf_bound} with $\Dh_L$
to characterize the response achieved by the FIR approximate controller $\Kh(L)$ on the true system $(\trueA,\trueB)$, giving the following:
\begin{align*}
J(\trueA,\trueB,\Kh(L)) &= \bightwonorm{\begin{bmatrix} Q^\frac{1}{2} & \\ & R^\frac{1}{2} \end{bmatrix} \begin{bmatrix}{\Phixh}\\{\Phiuh}\end{bmatrix}(I+\Dh_L)^{-1}} \\
& \leq \frac{1}{1-\widehat\gamma}\bightwonorm{\begin{bmatrix} Q^\frac{1}{2} & \\ & R^\frac{1}{2} \end{bmatrix} \begin{bmatrix}{\Phixh}\\{\Phiuh}\end{bmatrix}}.
\end{align*}
The inequality follows because $\hinfnorm{\Dh} \leq \widehat\gamma<1$. 

Denote by $(\tilde{\tf\Phi}_x,\tilde{\tf\Phi}_u,\tilde V,\tilde \gamma,\tilde \tau)$ the feasible solution constructed in Lemma \ref{lem:FIR_feasibility}.  Then,
\begin{align*}
  \frac{1}{1-\widehat\gamma}\bightwonorm{\begin{bmatrix} Q^\frac{1}{2} & \\ & R^\frac{1}{2} \end{bmatrix} \begin{bmatrix}{\Phixh}\\{\Phiuh}\end{bmatrix}}
  &\leq   \frac{1}{1-\tilde\gamma}\bightwonorm{\begin{bmatrix} Q^\frac{1}{2} & \\ & R^\frac{1}{2} \end{bmatrix} \begin{bmatrix}{\tilde{\tf\Phi}_x}\\{\tilde{\tf\Phi}_x}\end{bmatrix}}
  = \frac{1}{1-\tilde\gamma} {J_L(\Ah,\Bh,\tf K_c)} \\
 &\leq  \frac{1}{1-\tilde\gamma} {J(\Ah,\Bh,\tf K_c)} \leq \frac{1}{1-\tilde\gamma} \frac{1}{1-\|{\tf\Delta}\|_{2}} J(A_\star , B_\star , \tf K_c )\\
 &\leq  \frac{1}{1-\tilde\gamma} \frac{1}{1-\sqrt{2}\zeta_2} (1+M_\zeta) J(A_\star , B_\star , \tf K_\star ),
\end{align*}
where the first inequality follows from the optimality of  $(\Phixh,\Phiuh,\widehat V,\widehat \gamma)$, the equality and second inequality from the fact that $(\tilde{\tf\Phi}_x,\tilde{\tf\Phi}_u)$ are truncations of the response of $\tf K_c$ on $(\Ah,\Bh)$ to the first $L$ time steps. The second to last inequality follows from an application of~\eqref{eq:delta_perf_bound} with the roles of the nominal and true systems switched.
The final inequality follows by definition of the robustness cost gap and bounding $\|\tf\Delta\|_2$ by $\sqrt{2}\zeta_2$. We then simplify
\begin{align*}
\frac{1}{1-\tilde\gamma} \frac{1}{1-\sqrt{2}\zeta_2} &= \frac{1}{1-\frac{2\sqrt{2}\zeta_2}{1-\sqrt{2}\zeta_2}} \frac{1}{1-\sqrt{2}\zeta_2} = \frac{1}{1-3\sqrt{2}\zeta_2} \leq 6\sqrt{2}\zeta_2
\end{align*}

which follows for $\zeta_2<\frac{1}{6\sqrt{2}} $. Then finally,
\begin{align*}
\frac{J(\trueA,\trueB,\Kh(L)) - J_\star}{J_\star} 
&\leq 6\sqrt{2}(1+M_{\zeta})\zeta_2 + M_{\zeta} \:.
\end{align*}
\end{proof}

\begin{proof}[Proof of Lemma~\ref{lem:C_rho}]
Using Lemma~\ref{lem:freq_resp_composition} with $D(t)$ representing the frequency response elements of $(I-\tf\Delta)^{-1}$, we have that
\begin{align*}
\|\widehat\Phi_x^c(k)\|_p=\|\tf\Phi^c_x(I-\tf\Delta)^{-1}[k]\|_p
 &\leq \sum_{t=1}^{k} \|\Phi_x^\star(t)\|_p\|D(k-t)\|_p\\
& \leq \sum_{t=1}^{k} C_\star \rho_\star^t C_\Delta \rho_\Delta^{k-t} \leq C_\star C_\Delta k \max(\rho_\star,\rho_\Delta)^k\:.
\end{align*}
Then noting that
for any $x \in (0, 1)$, $\alpha > 0$, and $k \geq 1$,
\[\log (k x^k) = \log(k) + k\log(x) \leq \frac{1}{\alpha} k + \log(\alpha) + k\log(x), \]
which is true because $\log\left(\frac{k}{\alpha}\right)\leq \frac{k}{\alpha}$, we have for $p=2,\infty$
\[\|\widehat\Phi_x^c(k)\|_p \leq C_\star C_\Delta \alpha \left(e^{1/\alpha}\max(\rho_\star,\rho_\Delta)\right)^{k} \leq\frac{C_\star C_{\Delta}}{\log (\frac{\max(\rho_\star,\rho_\Delta)+1}{2\max(\rho_\star,\rho_\Delta)})}\left(\frac{\max(\rho_\star,\rho_\Delta) + 1}{2}\right)^k\:.\]
The final inequality follows from choosing $\alpha = \log\left(\frac{\max(\rho_\star,\rho_\Delta)+1}{2\max(\rho_\star,\rho_\Delta)}\right)^{-1}$.
\end{proof}

\section{Additive Disturbance Alternative}\label{sec:alternate_constraints}
In this section we discuss an alternative to the proposed robust control problem in~\eqref{eq:sls_lqr_robust} in which the constraints are tightened using a traditional additive disturbance approximation. 
We will contrast this additive reduction with the system-response reduction presented in Section~\ref{sec:robust_control}. Define the additive robust synthesis problem as
\begin{align}
\min_{\tf\Phi_x,\tf\Phi_u} ~ &\frac{1}{1-\gamma} J(\Ah, \Bh, \tf K) \label{eq:sls_lqr_robust_additive} \\
\mathrm{s.t.}~& \begin{bmatrix} zI - \Ahat & -\Bhat \end{bmatrix} \begin{bmatrix}\tf\Phi_x \\ \tf\Phi_u\end{bmatrix} = I,~~ \sqrt{2} \left\| \begin{bmatrix}\eps_{A,2}\tf\Phi_x \\ \eps_{B,2}\tf\Phi_u\end{bmatrix}\right\|_{\hinf} \leq \gamma, \notag\\
&G_x^\eps(\Phi_x;k) \leq b_x, ~~ G_u^\eps(\Phi_u;k) \leq b_u~~\forall ~k\geq 1\:. \notag
\end{align}
where $0\leq\gamma< 1$ is a fixed parameter and
\begin{align*}
\begin{split}
G_x^\eps(\tf\Phi_x;k)_j = G_x(\tf\Phi_x;k)_j +  (\eps_{A,\infty} r_x + \eps_{B,\infty} r_u) \|F_{x,j}^\top \Phi_x[k:1]\|_1\:,  \\
G_x^\eps(\tf\Phi_x;k)_j =  G_u(\tf\Phi_u;k)_j +  (\eps_{A,\infty} r_x + \eps_{B,\infty} r_u)\|F_{u,j}^\top \Phi_u[k:1]\|_1\:,
\end{split}
\end{align*}
where we define $r_x$ and $r_u$ as the $\ell_\infty$ radius of the constraint set, i.e. $r_x = \max_{ x\in \calX} \|x\|_\infty$ and similarly for $r_u$.


\begin{theorem}\label{thm:robust_constraint_satisfaction_additive}
Any controller designed from a feasible solution to the robust control problem~\eqref{eq:sls_lqr_robust_additive} for any $0\leq \gamma<1$ will stabilize the true system.
Furthermore,
the state and input constraints will be satisfied.
\end{theorem}

\begin{proof}
First, note that by the system norm constraint,
\begin{align*}
\begin{split}
&\|\Dh\|_{\hinf} \leq \sqrt{2} \left\| \begin{bmatrix}\eps_{A,2}\tf\Phi_x \\ \eps_{B,2}\tf\Phi_u\end{bmatrix}\right\|_{\hinf}<1\:,
\end{split}
\end{align*}
so the closed-loop system is stable.

Notice that the dynamics are given by
\[x_{k+1} = \Ahat x_k + \Bhat u_k + w_k + \Delta_A x_k + \Delta_B u_k\:.\]
We will make an additive disturbance approximation by defining $\tilde w_k = w_k + \Delta_A x_k + \Delta_B u_k$.
Though it is conservative, we can bound 
\begin{align*}
\|\tilde w_k\|_\infty &\leq \sigma_w + \max_{\substack{x\in\calX\\ \|\Delta_A\|_\infty \leq \eps_{A,\infty}}}\|\Delta_A x_k\|_\infty + \max_{\substack{u\in\calU\\ \|\Delta_B\|_\infty \leq \eps_{B,\infty}}}\|\Delta_B u_k\|_\infty \:.
\end{align*}
Now note that if $x_k\in\mathcal{X}$, we have
\[\max_{\substack{x\in\calX\\ \|\Delta_A\|_\infty \leq \eps_{A,\infty}}}\|\Delta_A x_k\|_\infty = \eps_{A,\infty}\cdot\max_{\substack{x\in\calX\\ \|\Delta\|_\infty \leq 1}}\|\Delta x_k\|_\infty = \eps_{A,\infty}\cdot\max_{x\in\calX}\| x_k\|_\infty \:.\]
Thus as long as state and input constraints are satisfied at time $k$, we have that $\|\tilde w_k\|_\infty \leq \sigma_w + \eps_{A,\infty}r_x + \eps_{B,\infty}r_u := \tilde\sigma_w$.

Then, notice that the constraint tightening is equivalent to replacing $\sigma_w$ in~\eqref{eq:sls_lqr} with $\sigma_w + \eps_{A,\infty}r_x + \eps_{B,\infty}r_u$. Then by Proposition~\ref{prop:sls_equivalence}, the inequality constraints $G_x^\eps$ and $G_u^\eps$ imply that the system satisfies the constraints for dynamics given by $(\Ahat,\Bhat)$ and any process noise bounded by $\tilde\sigma_w$.
Thus, we can recursively argue that state and input constraints are satisfied for all $k$ as long as they are satisfied for the initial state $x_0$.
\end{proof}

The next result establishes conditions under which this additive disturbance robust control is better than the system-response robust control.

\begin{proposition}\label{prop:additive_v_sls}
Let $\widehat \tau$ be an optimizer of~\eqref{eq:sls_lqr_robust_tau}.
The additive disturbance reduction is less conservative than the system-response based reduction if and only if
\[\max(\sigma_w, \|x_0\|_\infty)\frac{\widehat\tau}{1-\widehat\tau}> \eps_{A,\infty} r_x + \eps_{B,\infty} r_u\:.\]
\end{proposition}

From this statement, we can see that the system-response reduction suffers for large initial conditions, while the additive reduction suffers with the size of the constraint sets. For this reason, that additive disturbance reduction would not properly characterize the tradeoff between safety and exploration as plotted in Figure~\ref{fig:tradeoff}b.
However, we note that the condition in Proposition~\ref{prop:additive_v_sls} could be used as a switch in a control synthesis problem to decide which reduction to use.

\begin{proof}
Recall that $\tf w = x_0 z + \sum_{k = 0}^\infty w_k z^{-k}$.
The system-response reduction writes
 $$\begin{bmatrix}\tf x\\\tf u\end{bmatrix} = \begin{bmatrix}\Phixh\\ \Phiuh\end{bmatrix} (I+\widehat{\tf\Delta})^{-1}\tf w.$$
This expression can be viewed as expanding the disturbance by the uncertainty as
\[\tilde{\tf w}_{SLS} = (I+\widehat{\tf\Delta})^{-1}\tf w = \tf w + \widehat{\tf\Delta}(I+\widehat{\tf\Delta})^{-1}\tf w\]
In the robust constraints~\eqref{eq:sls_lqr_robust}, this additional uncertainty is bounded as
\[\|\widehat{\tf\Delta}(I+\widehat{\tf\Delta})^{-1}\tf w\|_\infty \leq \frac{\|\widehat{\tf\Delta}\|_{\lone}}{1-\|\widehat{\tf\Delta}\|_{\lone}}\max(\sigma_w,\|x_0\|_\infty)
\leq \frac{\widehat\tau}{1-\widehat\tau}\max(\sigma_w,\|x_0\|_\infty)\]

On the other hand, the additive disturbance approximation notices that
$x_{k+1} = \Ahat x_k + \Bhat u_k + w_k + \Delta_A x_k + \Delta_B u_k\:,$
and thus assigns $\tilde w_k = w_k + \Delta_A x_k + \Delta_B u_k$, i.e. 
$$\tilde{\tf w}_{add} = \tf w + \Delta_A \tf x + \Delta_B \tf u\:.$$
Then in the robust constraints~\eqref{eq:sls_lqr_robust_additive}, this additional uncertainty is bounded as
\[\|\Delta_A \tf x + \Delta_B \tf u\|_\infty \leq \eps_{A,\infty} r_x + \eps_{B,\infty} r_u\:.\]
Then the additive bound is worse than the system-response bound as stated. 
\end{proof}
Finally, we note that the sub-optimality result could be extended to controllers synthesized from minimizing~\eqref{eq:sls_lqr_robust_additive} over $\gamma$, and that the main difference would arise from the definition of the robust cost sub-optimality gap $M_\zeta$.
\section{Learning Results} \label{sec:payley_zygmund_proof}

First, we prove a simple small-ball result for random variables with finite fourth
moments.
\begin{proposition}
\label{prop:paley_zygmund}
Let $X \in \R$ be a zero-mean random variable with finite fourth moment, which satisfies the conditions
\begin{align*}
  \E[ X^4 ] \leq C (\E[X^2])^2 \:.
\end{align*}
Let $a \in \R$ be a fixed scalar and $\theta \in (0, 1)$. We have that
\begin{align*}
  \Pr\{ \abs{a + X} \geq \sqrt{\theta \E[X^2]} \} \geq (1-\theta)^2/\max\{4, 3C\} \:.
\end{align*}
\end{proposition}
\begin{proof}
First, we note that we can assume $a \geq 0$ without loss of generality,
since we can perform a change of variables $X \gets -X$.
We have that
$\E[(a+X)^4] = a^4 + 6 a^2 \E[X^2] + 4 a \E[X^3] +  \E[X^4]$.
Therefore, by the Paley-Zygmund inequality and Young's inequality, we have that
\begin{align*}
  \Pr\{ \abs{a + X} \geq \sqrt{\theta \E[X^2]} \} &\geq \Pr\{ \abs{a + X} \geq \sqrt{\theta \E[(a + X)^2] } \} \\
  &= \Pr\{ (a+X)^2 \geq \theta \E[(a+X)^2] \} \\
  &\geq (1-\theta)^2 \frac{(\E[(a+X)^2])^2}{\E[(a+X)^4]} \\
  &= (1-\theta)^2 \frac{ a^4 + 2a^2\E[X^2] + (\E[X^2])^2 }{ a^4 + 8a^2 \E[X^2] + 3\E[X^4] } \:.
\end{align*}
Now define the function $f(a)$ for $a \geq 0$ as
\begin{align*}
  f(a) = \frac{ a^4 + 2a^2 \mu + \mu^2 }{ a^4 + 8a^2 \mu + 3\beta } \:.
\end{align*}
Clearly $f(0) = \mu^2/(3\beta)$ and $f(\infty) = 1$.
Since by Jensen's inequality we know that $\mu^2 \leq \beta$, this means
$f(0) < f(\infty)$.
On the other hand,
\begin{align*}
  \frac{d}{da} f(a) = \frac{4 a \left(a^2+\mu \right) \left(\mu  \left(3 a^2-4 \mu \right)+3 \beta \right)}{\left(a^4+8 a^2 \mu +3 \beta \right)^2} \:.
\end{align*}
Assume that $\mu \neq 0$ (otherwise the claim is trivially true).
The only critical points of the function $f$ on non-negative reals is at $a=0$
and $a = \sqrt{\frac{4\mu^2 - 3\beta}{3\mu}}$ if $4\mu^2 \geq 3\beta$.
If $4\mu^2 < 3\beta$, then $a=0$ is the only critical point of $f$, so
$f(a) \geq f(0) = \mu^2/(3\beta) \geq 1/(3C)$.
On the other hand, if $4\mu^2 \geq 3\beta$ holds,
Some algebra yields that
\begin{align*}
    f(\sqrt{(4\mu^2 - 3\beta)/(3\mu)}) &= \frac{7\mu^2 - 3\beta}{16\mu^2 - 3\beta}
  \geq \inf_{\gamma \in [3,4]} \frac{7-\gamma}{16-\gamma} = 1/4 \:.
\end{align*}
The inequality above holds since $3\beta \in [3\mu^2, 4\mu^2]$.
Hence,
\begin{align*} f(a) \geq \min\{ 1/4, 1/(3C) \} = 1/\max\{4, 3C\} \:.
\end{align*}
\end{proof}

This gives us the main new technical ingredient needed to prove Theorem~\ref{thm:estimation}.
\begin{proposition}
\label{prop:small_ball}
Let $\mu \in \R^d$ and $M \in \R^{d \times d}$ be a full rank matrix.  Let $w
\in \R^d$ be a random vector such that its fourth moment is finite and
each coordinate $w_i$ is independent and zero-mean, i.e. $\E[w_i] = 0$.
Then, for any fixed $v \in \R^d$,
\begin{align*}
    \Pr_{w}\left( \abs{\ip{v}{\mu + M w}} \geq \sqrt{\lambda_{\min}(M \Sigma M^\T)/2}  \right) \geq \frac{1}{C\cdot C_w} \:,
\end{align*}
where $\Sigma = \E_{w}[ww^\T]$, $C$ is an absolute constant, and $C_w = \max_{1 \leq i \leq d} \E[w_i^4]/(\E[w_i^2])^2$.
In particular, we can take $C = 192$.
\end{proposition}
\begin{proof}
Fix a $q \in \R^d$.
By Rosenthal's inequality, we have that
$\E[ \abs{\ip{q}{w}}^4 ] \leq C \cdot C_w \E[\abs{\ip{q}{w}}^2]^2$
for an absolute constant--we can take $C=4$ (see e.g. \cite{ibragimov02}).
We now apply Proposition~\ref{prop:paley_zygmund} with $a = \ip{v}{\mu}$, $X =
\ip{M^\T v}{w}$, $C = 4 C_w$, and $\theta = 1/2$. The claim now follows.
\end{proof}
With Proposition~\ref{prop:small_ball} in place, the proof of Theorem~\ref{thm:estimation}
is identical to that of Proposition C.1 in Dean et al.~\cite{dean2018regret},
except in the final step of establishing the block martingale condition:
instead of using the small-ball calculation for Gaussian distributions we replace it with
the estimate provided by Proposition~\ref{prop:small_ball}.

\end{document}